\documentclass[aap]{imsart}
\RequirePackage{amsthm,amsmath,amsfonts,amssymb}
\RequirePackage[numbers]{natbib}
\RequirePackage[colorlinks,citecolor=blue,urlcolor=blue]{hyperref}
\RequirePackage{graphicx}
\usepackage[mathcal]{euscript}

\startlocaldefs

\newtheorem{theorem}{Theorem}[section]

\newtheorem{lemma}[theorem]{Lemma}
\newtheorem{thm}[theorem]{Theorem}

\newtheorem{define}[theorem]{Definition}
\newtheorem{assume}[theorem]{Assumption}
\newtheorem{example}[theorem]{Example}
\newtheorem{remark}[theorem]{Remark}

\usepackage[title]{appendix}

\usepackage{xspace}
\usepackage{enumerate}
\usepackage{booktabs,multirow,array}
\usepackage{longtable}
\usepackage{subfigure}
\usepackage{xcolor}

\usepackage{hyperref}
\hypersetup{
    colorlinks=true,
    linkcolor=blue,
    filecolor=blue,
    urlcolor=blue,
    citecolor=blue
}

\usepackage{graphicx}
\graphicspath{{figures/}}
\DeclareGraphicsExtensions{.pdf}

\usepackage{amssymb, amsmath, amsthm,dsfont,amsfonts,mathrsfs, mathtools}
\newcommand{\refdef}[1]{{Definition~\ref{#1}}}

\newcommand{\refapp}[1]{{Appendix~\ref{#1}}}

\newcommand{\refthm}[1]{{Theorem~\ref{#1}}}
\newcommand{\reflem}[1]{{Lemma~\ref{#1}}}

\newcommand{\refass}[1]{{Assumption~\ref{#1}}}

\newcommand{\pr}{\mathbb P}
\newcommand{\ex}{\operatorname{\mathbb E}}
\DeclareMathOperator{\var}{var}
\DeclareMathOperator{\cov}{cov}
\newcommand{\der}{\mathrm d}
\newcommand{\mc}{\mathcal}

\newcommand{\ind}{\operatorname{\mathds{1}}}

\newcommand{\mb}{\mathbf}
\newcommand{\del}{\partial}
\newcommand{\bs}{\boldsymbol}
\newcommand{\nn}{\nonumber}
\newcommand{\q}{\quad}

\newcommand{\mcS}{\mc{S}}
\newcommand{\mcP}{\mc{P}}

\newcommand{\mcN}{\mc{N}}
\newcommand{\mcT}{\mc{T}}
\newcommand{\wh}{\widehat}
\newcommand{\wt}{\widetilde}
\newcommand{\Rd}{\mathbb{R}^d}
\newcommand{\rmax}{r_{\textrm{max}}}

\DeclareMathOperator{\fin}{Fin}
\DeclareMathOperator{\vol}{vol}
\DeclareMathOperator{\bbQ}{\mathbb Q}
\DeclareMathOperator{\diam}{diam}
\def\bc{{\mathbf{c}}}
\def\cI{{\mathcal I}}

\def\sec{\setcounter{equation}{0}} 

\newcommand{\robert}[1]{{#1}}
\newcommand{\efe}[1]{{#1}}
\endlocaldefs

\begin{document}
\begin{frontmatter}

\title{Functional Central Limit Theorems for  Local Statistics of Spatial Birth-Death Processes in the Thermodynamic Regime}

\runtitle{FCLTs for Spatial Birth-Death Processes}

\begin{aug}
\author[A]{\fnms{Efe} \snm{Onaran}\ead[label=e1, mark]{efeonaran@campus.technion.ac.il}},
\author[A,B]{\fnms{Omer} \snm{Bobrowski}\ead[label=e2,mark]{omer@ee.technion.ac.il}}
\and
\author[A]{\fnms{Robert J.} \snm{Adler}\ead[label=e3,mark]{radler@technion.ac.il}}
\address[A]{Viterbi Faculty of Electrical and Computer Engineering,
Technion--Israel Institute of Technology\\\printead{e1,e2,e3}
}

\address[B]{School of Mathematical Sciences,
Queen Mary University of London}
\end{aug}

\begin{abstract}
We present normal approximation results at the process level for local functionals defined on dynamic Poisson  processes in $\mathbb{R}^d$. The dynamics we study here are those of a Markov birth-death process. We prove functional limit theorems in the so-called thermodynamic regime. Our results are applicable to several functionals of interest in the stochastic geometry literature, including subgraph and component counts in the random geometric graphs.

%


\end{abstract}

\begin{keyword}[class=MSC]
\kwd[Primary ]{60G55, 60F05}
\kwd[; secondary ]{60D05, 05C80}
\end{keyword}

\begin{keyword}
\kwd{Spatial birth-death process}
\kwd{functional central limit theorems}
\kwd{Ornstein-Uhlenbeck process}
\kwd{random geometric graphs}
\end{keyword}

\end{frontmatter}

\section{Introduction}
\sec

 Spatial birth-death processes are continuous time Markov chains with state spaces of counting measures in a metric space, typically $\Rd$. They have found extensive use in the stochastic simulation of applications where particles sporadically appear and disappear at random times \cite{MR2004226}. The first rigorous mathematical definition and treatment of these processes appeared in \cite{preston1975spatial}.

 Random geometric graphs, on the other hand, are static models generated by points randomly scattered in Euclidian space, for which pairs of points in close proximity are assumed to create  edges. The statistical properties of local, \efe{non-negative} functionals defined on random geometric graphs, such as numbers of edges or subgraphs isomorphic to a given graph, along with the techniques to derive them, are generally well understood \cite{MR1986198,MR1878288}. 
 
 In this paper, we establish \emph{functional} central limit theorems for a broad class of  local, \efe{non-negative} functionals defined on subsets of points in spatial birth-death processes. The functionals are parametrized by the size of their \emph{local support}, and our focus in this paper is on the regime where the expected number of particles in shrinking supports converges to a constant as the overall density of particles goes to infinity. This is generally called the \emph{thermodynamic regime} in the random geometric graph literature. We prove that subject to mild assumptions,  normalized sums of these spatial functionals (summed over all subsets of the underlying particles alive at a given time) converge weakly to weighted superpositions of independent Ornstein-Uhlenbeck processes. In particular, this superposition is infinite when the functionals satisfy a certain neighbourhood vacancy condition, an example of which is the isolated subgraph (\emph{component}) count. Further examples of practical interest  will be given as special cases at the end of Section \ref{sec2}.
 
As for our methods, we rely heavily on the modern theory of normal approximation for stabilizing functionals. Historically, central limit theorems in stochastic geometry have been based either on normal approximation theory for martingale differences \cite{MR1986198} or on the classical Stein method \cite{MR2201885}. Recent developments for Poisson space functionals, fusing the Malliavin calculus and the Stein method \cite{MR3520016} led to normal approximation theorems easily applicable to point processes in more general spaces, while also providing improved convergence rates \cite{lachieze2019normal}. We extend the applications of this new theory by modelling the temporally evolving spatial birth-death process as a static \emph{marked} point process. We employ the marked point process representation also to calculate the limiting covariance and bounds on the higher moments needed to establish the limit theorems at the process level. 

Our results echo other findings in related areas. Central limit theorems for functionals over dynamic complexes were considered in \cite{Thoppe2016OnComplexes}, where the authors showed that the normalized number of $k$-cliques in a dynamic Erdős-Rényi graph with $n$ vertices 
 converges weakly to an Ornstein-Uhlenbeck process. Consequently, similar CLTs apply to the topological invariants (e.g. Betti numbers) of the corresponding random flag complex.
In \cite{Owada2021Limit} these results were extended from the random flag complex to the generic, multi-parameter, random simplicial complex model introduced in \cite{Costa2017}.
 In the geometric setting, \cite{owada2017functional} considers the extreme value distribution of subgraph counts in random geometric graphs as a process of varying proximity threshold. The normalized subgraph count in \cite{owada2017functional} was proved to converge weakly to a superposition of a new class of Gaussian processes in an extreme-value regime that is equivalent to thermodynamic regime we study here. These results were later extended to Betti numbers in \cite{MR3847974}, complementing the earlier central limit theorem in \cite{Yogeshwaran2014RandomRegime} for the static case. Interestingly, infinite superpositions of Gaussian processes show up in the limit theorems in \cite{MR3847974} as well. 
  

Finally, we note that functional CLTs for spatial birth-death processes  were obtained in related scenarios, although for considerably different settings.
In \cite{qi}, similar birth-death processes are studied, with more general rate functions, albeit for functionals that can only be applied to \emph{single points}. In \cite{penrose_continuum}, generic local functionals of point processes are studied, using the martingale central limit theorem \cite{penrose_2001}. However, the birth-death dynamics there are different than here in two important ways. Firstly, the dynamics there require the local jump rates to be \emph{uniformly bounded}, which fails for dynamic Poisson processes with independent  death rates, as the local number of points is unbounded. Secondly, the model assumptions in \cite{penrose_2001} make it applicable only to scenarios with spatially uniform point density. 
More recently, functionals of random tessellations of the pure-birth processes (the ``dead leaves'' model) were shown to converge to Ornstein-Uhlenbeck process in \cite{penrose_leaves}.

\robert{In Section \ref{sec2} we turn to the precise formulations of our results, and more carefully describe the  applications. Section \ref{sec3} provides preliminaries to be used in the proofs. Finally, Sections \ref{proof1} and \ref{proof2} will provide the proofs for Theorems \ref{OU} and \ref{OU2}, respectively.}

\section{Main Results}
\label{sec2}
\sec 

The point process model that we study in this paper is as follows. Let $\eta_n(\cdot)$ be a continuous-time Markov process on the state space $\fin(\Rd)$ consisting of all finite subsets of $\Rd$. The infinitesimal generator of $\eta_n$, denoted by  $L$, is defined via
\begin{equation}\label{eqn:generator}
Lh(\eta) := n \int_{\Rd} D_x h(\eta) \bbQ(dx) -  \sum_{y\in \eta} D_y h(\eta\backslash \{y\}),\quad \eta\in \fin(\Rd),
\end{equation}
where $h:\fin(\Rd)\to\mathbb{R}$ is a bounded, real valued, function,  $D_x$ is the add-one cost operator
\[
    D_xh(\eta)\coloneqq h(\eta\cup\{x\}) - h(\eta), 
\] 
and 
 $\bbQ$ is a probability measure on $\Rd$, with density $q$.
We will assume that $q$ is bounded and  continuous almost everywhere, and
\begin{align*}
    \|q\|_{\infty}\coloneqq \sup_{x\in\Rd} q(x)<\infty.
\end{align*}

In  words, the probability for a single point to be added to the process (or a ``birth'') during the time interval $[t,t+\Delta t]$ and inside a Borel set $A\subset \Rd$ is $n\bbQ(A)\Delta t + o(\Delta t)$. The probability that an existing point is removed from the process (a ``death'') during such an interval is $\Delta t + o(\Delta t)$. All other possible changes occur with probability $o(\Delta t)$. In addition,  changes that occur in disjoint time intervals are independent.

These processes are known as spatial birth-death processes with constant birth and death rates. 
We refer the reader to \cite{MR2371524} for a comprehensive theory of spatial point processes. An important fact which will be essential in this paper is that the stationary distribution for the Markov process defined by \eqref{eqn:generator} is that of a Poisson point process in $\Rd$, with intensity measure $n\bbQ$ (see 
 \cite{preston1975spatial,van2000markov}). Throughout, we will assume that $\eta_n(0)$ has this stationary distribution, from which it follows that $\eta_n$ is stationary on $[0,\infty)$.
 
Next, we define the family of functionals  of interest to us.
Let  $\xi_{k,r}:\fin(\Rd)\to \mathbb{R}^+\coloneqq [0,\infty)$ be a \efe{non-negative} subset functional parametrized by $r> 0$, such that $\xi_{k,r}(\mcS) = 0$ if $|\mcS| \ne k$, where
 $|\mcS|$ denotes the number of points in $\mcS$.
 
 We assume that $\xi_{k,r}$ satisfies the following assumptions. 

\begin{assume}[Translation and scale invariance]\label{inv}
For any $\mcS\in \fin(\Rd)$, $\alpha>0$, $x\in\Rd$, and $r>0$
\begin{align*}
    \xi_{k,r}(\alpha \mcS+x) = \xi_{k, r/\alpha}( \mcS),
\end{align*} 
where set addition and scalar multiplication are defined in the natural way.

\end{assume}

\begin{assume}[Localization]\label{localass}
There exists a constant $r_{\textrm{max}}>0$ such that $\xi_{k,1}(\mcS) = 0$ if the diameter of the set satisfies $\diam(\mcS)> \rmax$.
\end{assume}

\begin{assume}[Boundedness]\label{boundedass}
The function $\xi_{k,1}$ satisfies
\[
\|\xi_{k,1}\|_{\infty} \coloneqq \sup\limits_{\mcS\in \fin(\Rd)} |\xi_{k,1}(\mcS)|<\infty.\]
\end{assume}
Lastly, to avoid trivialities, we make the following \emph{feasibility} assumption on $\xi_{k,1}$.

\efe{\begin{assume}[Feasibility] \label{feas} There exists a nonempty open set $\mcN \subset (\mathbb R^d)^{k-1}$, such that
the function $\xi_{k,1}$ satisfies for all $\mb{y}\in \mcN$,
    \begin{equation}
     \xi_{k,1}(0, \mb y) >0.  
     \label{feaseq0}
    \end{equation}
\end{assume}}

Recall  that  $\xi_{k,r}$ was originally defined as a function whose domain is $\fin(\Rd)$. However, in Assumption \ref{feas} it is treated as a function on $(\mathbb{R}^d)^k$, invariant under permutations among $d$-dimensional partitions. We shall use this minor notational ambiguity throughout the paper. The meaning will always be clear from the context, and it will save us redundant notation.

Given a functional $\xi_{k,r}$  satisfying  Assumptions \ref{inv}-\ref{feas}, we will be interested in two types of processes, both constructed via the Markov process $\eta_n$ defined above. The first one is defined as
\[
     f_{n}(t) \coloneqq \sum_{\mcS \subset \eta_n(t)} \xi_{k,r} (\mcS), 
\]
which is merely the sum of the functional over all subsets. Before we define the latter, we assume that $N_r(\mcS)$ are Lebesgue measurable sets satisfying the following.
\begin{assume}[Affine invariance]\label{linF}
For any $\mcS\in \fin(\Rd)$, $\alpha,r>0$, and $x\in\Rd$,
\begin{align*}
     N_{\alpha r}(\alpha \mcS +x) = \alpha N_{r}( \mcS ) +x.
 \end{align*}
\end{assume}
\begin{assume}[Localization]\label{loc2}
There exists $\beta_k>0$ such that 
 for all $\mcS\in \fin(\Rd)$, with $|\mcS| = k$, 
 \begin{align*}
     \diam\big(N_r(\mcS)\big)&\leq  \beta_k\max\big(r,\diam(\mcS)\big).
 \end{align*}
\end{assume}

\efe{\begin{assume}[Joint feasibility] \label{jointfeas} There exists a nonempty open set  $\tilde{\mcN}\subset (\Rd)^{k-1}$, such that
the functions $\xi_{k,1}$ and $N_{1}$ satisfy for all $\mb{y}\in \tilde{\mcN}$,
    \begin{equation}
     \xi_{k,1}(0, \mb{y})\vol\left[N_1(0,\mb{y})\right] >0,   
     \label{feaseq}
    \end{equation}
where $\vol$ denotes  Lebesgue measure.  
\end{assume}}

With these assumptions, we define the second process of interest to us as
\begin{equation*}
  F_{n}(t) \coloneqq \sum_{\mcS \subset \eta_n(t)} \xi_{k,r} (\mcS) \ind\big\{ (\eta_n(t)\setminus\mcS)\cap N_r(\mcS)=\varnothing \big\},
\end{equation*}
which is an \emph{exclusive} sum where we require the neighborhood $N_r(\mcS)$ to be empty.


The limit theorems we prove in this paper will in fact be applied to the normalized versions of the processes $f_n(t)$ and $F_n(t)$, given as
\[
\bar{f}_{n}(t) \coloneqq \frac{{f}_{n}(t) - \ex[{f}_{n}(t)] }{\sqrt{\text{var} [f_{n}(t)]} }, 
\quad \text{and}\quad 
\bar{F}_{n}(t) \coloneqq \frac{{F}_{n}(t) - \ex[{F}_{n}(t)] }{\sqrt{\text{var} [F_{n}(t)]} }, 
\]
where we implicitly assume the Poisson-Markov model for $\eta_n$ described at the beginning of this section.

In order to state our results succinctly, we use $\{\mc{U}_{\kappa}(t):t\geq 0\}$, for some $\kappa>0$, to denote the stationary, Gaussian, zero mean, Ornstein-Uhlenbeck (OU) process with  covariance function $$\text{cov}[\mc{U}_{\kappa}(t_1), \mc{U}_{\kappa}(t_2)] = e^{-\kappa|t_1-t_2|}.$$
For a given sequence $\bc = (c_1,c_2,\ldots)$ we define the following process
\[
\mc{U}_{\bc}(t) := \sum_{j=1}^\infty c_j\mc{U}_j(t),
\]
 a weighted superposition  of \emph{independent} OU processes. By $\|\bc\|$ we refer to the $\ell^2$ norm.  
 

\begin{thm}\label{OU}
 Suppose that $\lim_{n\to\infty} nr^d = \gamma \in (0,\infty)$ \efe{and that Assumptions \ref{inv}-\ref{feas} on $\xi_{k,r}$ hold}. Then there exists $\bc$ (depending on $\bbQ,k,\xi_{k,1}$ and $\gamma$), satisfying $c_j> 0$ for all $1\le j \le k$, $c_j = 0$ for all $j>k$,  and  $\|\bc\| = 1$,  such that
 the processes $\{\bar{f}_{n}(t):t\geq 0\}$ converge weakly to $\{\mc{U}_{\bc}(t) :t\geq 0\}$. 
\end{thm}

\begin{thm}\label{OU2}
 Suppose that $\lim_{n\to\infty} nr^d = \gamma \in (0,\infty)$ and \efe{that Assumptions \ref{inv}-\ref{jointfeas} on $\xi_{k,r}$ and $N_r$ hold}. Then there exists $\bar\bc$ (depending on $\bbQ,k,\xi_{k,1},N_1$ and $\gamma$), satisfying $\bar c_j>0$ for all $j$, and $\|\bar\bc\|=1$, such that 
 the processes $\{\bar{F}_{n}(t):t\geq 0\}$, converge weakly to $\{\mc{U}_{\bar\bc}(t):t\geq 0\}$. Furthermore, the limiting process has almost surely continuous paths. 
\end{thm}
In both theorems, weak convergence is considered in the space of c\`adl\`ag functions on $[0,\infty)$ equipped with the usual Skorokhod metric. 
Note that the claim in Theorem \ref{OU2} about continuity of the limit process is true in the setting of Theorem \ref{OU} as well. However, since the limit is a 
 \emph{finite} superposition of OU processes, continuity here is straightforward. This is not the case for the infinite sum.
 
We conclude this section with two examples for which the above two theorems hold.
The first comes from the area of 
random geometric graphs, and the second is related to random distance functions -- a powerful tool in the study of the topology of random simplicial complexes.

 \begin{example}[Geometric graphs]
 \label{example-graphs}
Let $G(\eta, r)$ be a geometric graph built over a finite $\eta\subset \Rd$ with distance parameter $r$, i.e.\ an edge is placed between any two points in $\eta$ no further than $r$ apart. Define
 $$\xi_{k,r}(\mcS) = \ind\{\textrm{$G(\mcS,r)$ is a $k$-clique}\}.$$
Then $\xi_{k,r}$ satisfies Assumptions \ref{inv}-\ref{boundedass}, and ${f}_n(t)$ counts the number of $k$-cliques formed by $G(\eta_n(t),r)$. In addition, if we take
\[
    N_r(\mcS) := \bigcup_{x\in \mcS} B(x,r),
\]
where $B(x,r)$ is the ball of radius $r$ around $x$, 
then $N_r(\mcS)$ satisfies Assumptions \ref{linF} and \ref{loc2}, and we have that $F_n(t)$ counts the number of $k$-clique \emph{components} in the graph $G(\eta_n(t),r)$.
Note that, in a similar fashion, by choosing $\xi_{k,r}$ appropriately we can set $f_n(t)$ to count the number of copies of any feasible graph $\Gamma$ on $k$ vertices, and $F_n(t)$ to count the number of components isomorphic to $\Gamma$.
\end{example}


\begin{example}[Distance functions]
\label{example-distancefunction}
Given a finite $\mc{A}\subset \Rd$  define the distance function 
$$d_{\mc{A}}(x) = \min_{y\in\mc{A}} \| x-y\|.$$ 
In \cite{gershkovich1997morse} it was shown that critical points for the function $d_{\mc{A}}$, of Morse index $k$, are generated by subsets $\mcS\subset {\mc{A}}$ for which $|\mcS| = k+1$ and satisfying the following two conditions, in which $B(\mcS)$ is the smallest open ball containing $\mcS$ on its boundary:

(a) The center of $B(\mcS)$ lies inside the $k$-dimensional simplex spanned by $\mcS$,

(b) $B(\mcS)\cap{\mc{A}} = \emptyset$.

\noindent Define $\xi_{k+1,r}$ to be the indicator that (a) holds (for some range of distances determined by $r$), and take $N_r(\mcS) = B(\mcS)$. Then $F_n(t)$ counts the number of critical points of index $k$, for the dynamic distance function $d_{\eta_n(t)}(x)$. The analysis of such critical points is of its own interest, but more importantly it plays a key role in  the study of the topology of random \v{C}ech complexes \cite{bobrowski2019homological,MR3280987}.   
\end{example}


\section{Preliminaries}
\label{sec3}
\sec
Before proving the main theorems, we present a few definitions and recall  existing theorems that will be used later. 

\subsection{A Functional CLT}
The proofs of Theorems \ref{OU} and \ref{OU2} will rely on the following functional version of the CLT. Let $D_\mathbb{R}[0,\infty)$ denote the space of c\`adl\`ag functions on $[0,\infty)$ equipped with the  Skorokhod metric, denoted by $\rho$. We say that a family of $(D_\mathbb{R}[0,\infty), \rho)$-valued stochastic processes, $\{X_n(\cdot)\}_{n=1}^\infty$,  is tight, if,  for every $\varepsilon>0$, there exists a compact set $\Gamma_\varepsilon\subset D_\mathbb{R}[0,\infty)$ for which 
\begin{align*}
    \inf_{n\geq 1}\pr[X_n\in \Gamma_\varepsilon]\geq 1-\varepsilon.
\end{align*}
Theorem 2.2 in \cite{Thoppe2016OnComplexes}, adapted from \cite{ethier2009markov}, will play an essential role in our proofs. For ease of reading, we split it into two statements.

\begin{thm}[Theorem 2.2 in \cite{Thoppe2016OnComplexes}]\label{mainthm}
  Let $\{X_n(\cdot)\}_{n=1}^\infty$ be a tight sequence in $(D_\mathbb{R}[0,\allowbreak\infty), \rho)$. If the finite dimensional distributions of $X_n$ converge in distribution to those of the process $Z$, then the processes $X_n$ converge in distribution to $Z$. 
 \end{thm}
 \begin{thm}[Theorem 2.2 in \cite{Thoppe2016OnComplexes}]\label{mainthm2}
 The sequence $\{X_n(\cdot)\}_{n=1}^\infty$ is tight in $(D_\mathbb{R}[0,\allowbreak\infty), \rho)$ if the following holds.
 \begin{itemize}
     \item[\textnormal{(C1)}] There exists $\Upsilon>0$ such that
 \[\lim_{\delta\to 0} \limsup_{n\to\infty}\ex\big[|X_n(\delta) -X_n(0) |^\Upsilon\big]=  0.\]
 
     \item[\textnormal{(C2)}] For every $T_0>0$, there exist constants $\Upsilon_1>0$, $\Upsilon_2>1$, and $C>0$ such that for all $n$, $0\leq t\leq T_0+1$, and $0\leq h\leq t$,
 \[ \ex\big[|X_n(t+h)-X_n(t)|^{\Upsilon_1} |X_n(t)-X_n(t-h)|^{\Upsilon_1} \big]\leq Ch^{\Upsilon_2}. \]
 \end{itemize}
\end{thm}

\subsection{Useful lemmas}
We will need two lemmas (\ref{Palmhigh} and \ref{lemgeom}) which
generalize common counting techniques for Poisson processes, when the objects being counted have an intricate structure of intersections. The following notation is needed for their statements. 

\begin{define}\label{defn:pattern}
Let $\cI_\ell$ denote a collection of natural numbers indexed by nonempty subsets of $[\ell] = \{1,\ldots,\ell\}$, i.e. $\cI_\ell = (I_J)_{J\subset [\ell], J\ne\emptyset}$. Given a sequence of sets $\mcS_{1},\ldots,\mcS_{\ell} \in \fin(\Rd)$, suppose that for all nonempty $J\subset[\ell]$ we have
\[
\left| \left(\bigcap_{j\in J} \mcS_j \right)\cap\left(\bigcap_{j\not\in J} (\mcS_j)^c \right)\right| = I_J,
\]
where $(\mcS_j)^c$  is  the complement of $\mcS_j$.
In this case we say that $\mcS_{1},\ldots,\mcS_{\ell} $ obey the \emph{intersection pattern} $\cI_\ell$, and denote this by
$(\mcS_{1},\ldots,\mcS_{\ell}) \in \cI_\ell$.
\end{define}

In what follows, we will typically  write $\cI$ for $\cI_\ell$, unless $\ell$ is explicitly required. Set 
\[|\cI| := \sum_{J} I_{J}  = |\mcS_1\cup\cdots\cup \mcS_{\ell}|.\]
Fixing $\cI$, and given a sequence of points ${\mb x} = (x_1,\ldots,x_{|\cI|})$ in $\Rd$, 
let
\[
\Pi_{\cI}(\mb x) \coloneqq (\mcT_1,\ldots,\mcT_\ell) 
\]
be a splitting of the tuple $\mb x$ into  $(\mcT_1,\ldots,\mcT_\ell)\in \cI$, in an arbitrary but fixed manner (e.g., according to lexicographic ordering). Finally, let $\#_{\cI}$ denote the number of connected components in the intersection graph of $\cI$, i.e., the graph with $\ell$ nodes, where each node pair $(i,j)$ is connected by an edge if and only if  $|\mcT_{i}\cap \mcT_{j}|>0$. With regard to the notation used throughout the paper, when there is no room for ambiguity, we allow a tuple to act as a set when under a set action, such as union or intersection. 

 The first lemma is a generalization of the well-known Mecke formula for Poisson point processes. The second lemma generalizes  asymptotic results for subgraph counting in random geometric graphs. Special cases of both lemmas can be found in \cite{MR1986198}. 

\begin{lemma}\label{Palmhigh}
Let $\mcP_{n}$ be a Poisson point process on $\mathbb{R}^d$ with intensity function $n q(x)$, where $q(x)$ is a bounded probability density function on $\mathbb{R}^d$. Let $h( \bar{\mcS}, \mcP_n) $ be a bounded measurable real function with $\bar{\mcS} \coloneqq(\mcS_{1},\ldots,\mcS_{\ell})$. Then,
\[
     \ex\Big[\sum_{\mcS_{1} \subset \mcP_{n}} \cdots \sum_{\mcS_{\ell} \subset \mcP_{n}} h(\bar{\mcS}, \mcP_n)\ind\{ \bar{\mcS} \in \cI\}\Big]  =  \frac{n^{|\cI|} \ex \big[h\big(\Pi_{\cI}(\mb X), \mb{X}\cup\mcP_n\big) \big] }{\prod_{J}I_{J}!},   
\]
where $\mb X$ is a tuple of $|\cI|$ iid points in $\Rd$, with density $q$, and independent of $\mcP_n$.  

\end{lemma}

\begin{proof}
Conditioning on the number of points in $\mcP_n$, we have, for $m\ge |\cI|$,
  \begin{align} \label{exbigsum}
  \begin{split}
     \ex \Big[\sum_{\mcS_{1} \subset \mcP_{n}} \cdots \sum_{\mcS_{\ell} \subset \mcP_{n}} &h(\bar{\mcS}, \mcP_n)\ind\{ \bar{\mcS} \in \cI\} \Big| |\mcP_{n}|=m\Big] \\&=
     \frac{m!}{(m-|\cI|)!\prod_J I_J!}
          \ex \left[h\big(\Pi_{\cI}(\mb X), \mb{X}^m\big) \right],
    \end{split}
 \end{align}
 where $\mb X^m = (X_1,\ldots, X_m)$ is a sequence of $m$ iid points in $\Rd$, with density $q$, and, similarly, $\mb X = (X_1,\ldots, X_{|\cI|})$. \efe{Note that the factor in front of the expectation in \eqref{exbigsum} comes from the number of ways that $m$ points can be partitioned into $\ell$ sets, that obey the intersection pattern $\cI$.
Using the law of total expectation, and then making the change of variables $m-|\cI|\to m $, we obtain}
\[
\ex\Big[\sum_{\mcS_{1} \subset \mcP_{n}} \ldots \sum_{\mcS_{\ell} \subset \mcP_{n}} h(\bar{\mcS}, \mcP_n)\ind\{ \bar{\mcS} \in \cI\} \Big] 
     = \frac{ n^{|\cI|} }{\prod_{J}I_{J}!} \sum_{m=0}^{\infty} \frac{e^{-n}n^m}{m!} \ex \big[h\big(\Pi_{\cI}(\mb X),\mb{X}\cup\mb{\hat X}^{m} \big) \big],
\]
where $\mb{\hat X}^{m}$ is an independent copy of $\mb{X}^m$. This completes the proof.
\end{proof}

The second lemma governs the asymptotic scaling of  products of functionals $\xi_{k,r}$, where the parameters $k$ might vary over the various terms in the product.
 

\begin{lemma}\label{lemgeom}
Let $\cI = \cI_\ell$ be an intersection pattern, such that if $(\mcS_{1},\ldots,\mcS_{\ell}) \in \cI_\ell$ then  $|\mcS_j| = k_j \ge 2$ for every $1\le j \le \ell$.
Let $\mb X$ be a sequence of $|\cI|$ iid points in $\Rd$ with density $q$, and let $(\mcT_1,\ldots,\mcT_\ell) = \Pi_{\cI}(\mb{X})$. 
 Then there is a constant $c\geq 0$ such that
 \[
     \lim_{r\to 0} r^{-d(|\cI| - \#_{\cI})} \ex\left[\prod_{j=1}^\ell \xi_{k_j,r}(\mcT_j) \right] = c  .
\]
\end{lemma}
\begin{proof}
Assume the intersection graph of $\cI$ is connected. Setting
\begin{align*}
    {\xi}_{r} \big(\Pi_{\cI}(\mb{X})\big)\coloneqq \prod_{j=1}^\ell \xi_{k_j,r}(\mcT_j),
\end{align*}
note that ${\xi}_{r}$ inherits the translation and scale invariance, localization and the boundedness of  the $\xi_{k_j,r}$. Next, note that
\[
\begin{split}
   \ex\left[\prod_{j=1}^\ell \xi_{k_j,r}(\mcT_j) \right] =& \int_{\Rd} q(x_1) \int_{\left(\Rd\right)^{|\cI|-1}} {\xi}_{r} \big(\Pi_{\cI}(\mb x)\big)\bigg( \prod_{i=2}^{|\cI|} q(x_i)\der x_i\bigg) \der x_1.
\end{split}
\]
Using the change of variables $x_i = x_1+ry_{i-1}$ for $2\leq i\leq |\cI|$, and  the translation and scale invariance of ${\xi}_{r}$,
\[
\frac{\ex\big[\prod_{j=1}^\ell\xi_{k_j,r}(\mcT_j) \big]}{r^{d(|\cI|-1)}} =
    \int_{\left(\Rd\right)^{|\cI|}}  {\xi}_{1} \big(\Pi_{\cI}(0,\mb y)\big) q(x_1) \prod_{i=1}^{|\cI|-1} q(x_1+ry_i)\der y_i \der x_1, 
\]
where $(0,\mb y)=(0,y_1,\ldots,y_{|\cI|-1})$. 
Note, due to the connectivity assumption on the intersection graph of $\cI$, the absolute value of the integrand on the right hand side can be  bounded above by
\begin{align*}
    \left(\max_{1\leq j\leq \ell}\|\xi_{k_j}\|_{\infty}\right)^\ell \big(\|q\|_{\infty}\big)^{|\cI|}\prod_{i=1}^{|\cI|-1}\ind\{y_i\in B(0,\ell\rmax)\} ,
\end{align*}
which is integrable. Therefore, using the dominated convergence theorem and the assumed almost everywhere continuity of $q$, we have 
\[
    \lim_{r\to 0} \frac{\ex\big[\prod_{j=1}^\ell\xi_{k_j,r}(\mcT_j) \big]}{r^{d(|\cI|-1)}} =\int_{\Rd}[q(x)]^{|\cI|}\der x \int_{(\Rd)^{|\cI|-1}} {\xi}_{1} \big(\Pi_{\cI}(0,\mb y)\big) \prod_{i=1}^{|\cI|-1} \der y_i,
\]
which is a non-negative constant due to the boundedness of $q(x)$ and the assumptions on $\xi_{k_j,1}$.

Finally, for general $\cI$ (i.e., where the intersection graph is not necessarily connected), the proof follows by calculating the expectation for the components of the intersection graph separately and using the independence of the functions $\xi$ that operate on different components.
\end{proof}

\section{Proof of \refthm{OU}}\label{proof1}\sec
The proof of \refthm{OU} will be established through a sequence of lemmas, divided into four main steps: (1) characterizing the asymptotic behavior of the mean and variance of $f_{n}(t)$, (2) proving convergence of the finite dimensional distributions, (3) calculating the limiting covariance function, and (4) proving tightness. \robert{Throughout the proofs we  implicitly assume the model described in the previous section, the assumptions on the density $q$, and Assumptions \ref{inv} to \ref{jointfeas}.}

\subsection{First and Second Moments}

Since the process $\eta_n(t)$ is stationary, applying  Mecke's Formula (\reflem{Palmhigh}, with $\ell=1$) yields
\begin{align}
 \ex[f_{n}(t)] 
 = \frac{n^{k}}{k!}\alpha_{k}, \label{125}   
\end{align}
 where we use $\alpha_k$ as  shorthand notation for
\[
    \alpha_k\coloneqq \ex[\xi_{k,r}(\mb X)],
\]
  and $\mb X\coloneqq (X_1,\ldots,X_k)$ is a tuple of  iid points with density $q$. 
  Next, we have
\begin{align*}
    \ex[f_{n}(t)^2] &=  \ex\Bigg[ \sum_{i=0}^{k} \sum_{\substack{\mcS_1, \mcS_2\subset \eta_n(t)\\|\mcS_1\cap \mcS_2|=i}}    \xi_{k,r}(\mcS_1)\xi_{k,r}(\mcS_2)\Bigg].
\end{align*}
Using \reflem{Palmhigh} we obtain
\begin{align}\label{126}
    \ex[f_{n}(t)^2]&= \sum_{i=0}^{k} \frac{n^{2k-i}}{i!((k-i)!)^2} \alpha_{k,i}, 
\end{align}
where 
\begin{align}\label{probdef}
    \alpha_{k,i}&\coloneqq \ex\left[\xi_{k,r}(\mb X)\xi_{k,r} (\mb X' )\right],
\end{align}
and $\mb X'\coloneqq (X_{k-i+1},\ldots, X_{2k-i})$ are iid points with density $q$ ($i$ of them intersect with the points in $\mb X$). Note that $\alpha_{k,0} = \alpha^2_{k} $ for any $k\geq 2$.
From \eqref{125} and \eqref{126} we can calculate the variance as follows,
\[
\var[f_{n}(t)] = \frac{n^{k}}{k!}\left[\sum_{i=0}^{k} \binom{k}{i}  \frac{n^{k-i}}{(k-i)!} \alpha_{k,i} -\frac{n^{k}}{k!} \alpha_{k}^2\right]
    =\frac{n^{k}}{k!}\sum_{i=1}^{k} \binom{k}{i}  \frac{n^{k-i}}{(k-i)!} \alpha_{k,i}.
\]
It follows that 
 \begin{align*}
     \var[f_{n}(t)] \geq  \frac{n^{k}}{k!}\alpha_{k,k}.
 \end{align*}
Applying \reflem{lemgeom} (with $\ell = 1$), we have that  
\[\lim_{r\to 0} \alpha_{k,k} r^{-d(k-1)} = \int_{\Rd}q(x)^{k}\der x \int_{(\Rd)^{k-1}} \xi_{k,1}^2(0, \mb y) \prod_{i=1}^{k-1} \der y_i>0, \] 
due to \eqref{feaseq}. Recalling that $nr^d\to \gamma >0$,  we conclude that
\begin{equation}
\kappa_k := \liminf_{n\to\infty}\frac{\var[f_{n}(t)]}{ n} >0. \label{eqvarbound}
\end{equation}

\subsection{Finite Dimensional Distributions}

In this subsection, we prove   convergence of the finite dimensional distributions of the $\bar{f}_{n}(t)$.

\begin{lemma}\label{finite}
Suppose that $\lim_{n\to\infty} nr^d = \gamma \in (0,\infty)$. Then the finite dimensional distributions of $\bar{f}_{n}(t)$ converge to a multivariate Gaussian distribution.
\end{lemma}

To prove Lemma \ref{finite}, we use the CLT result proven in \cite{lachieze2019normal} for functionals of \emph{marked} Poisson processes. The setting considered in \cite{lachieze2019normal} consists of a Poisson  process $\mcP_n$ on $\Rd$ with the intensity measure $n\mathbb Q$.
Each point $x\in \Rd$ in the process is further assumed to be associated with an independent and identically distributed mark, $M_x$, in the mark space $\mathbb M$. Let the space of a single marked point in $\Rd$ be $\wh{\mathbb X}\coloneqq  \Rd\times \mathbb M$, and we use $\wh{x}$ for the  pair $(x,M_x)$. We denote the marked process by $\wh \mcP_n$. Before we present the central limit theorem of \cite{lachieze2019normal}, we need a few definitions. Note that the definitions we present here are simplified versions of similar concepts in \cite{lachieze2019normal} that suffice for our purposes.

 \begin{define}[Decaying stabilization]\label{decstab}
 Let $R_n\geq 0$ with $\limsup\limits_{n\to\infty} n(R_n)^d <\infty$. We say that $R_n$ is a radius of decaying stabilization for the (local) function  $h_n: \wh{\mathbb{X}}\times \fin(\wh{\mathbb X})\to \mathbb R$, if, for all $(x,M_x)=\wh{x}\in \wh{\mathbb{X}}$ and $\wh{\mc Y} \in \fin(\wh{\mathbb X})$,
\[
     h_n \big(\wh x, (\wh{\mc Y} \cup \{\wh x\})\cap \wh B(x, R_n) \big) =h_n \big(\wh x, \wh{\mc Y} \cup \{\wh x\}\big) 
\]
 where $\wh B(x,\lambda)\coloneqq B(x,\lambda)\times \mathbb M$ for $x\in\Rd$ and $\lambda>0$.
 \end{define}
Simply put, Definition \ref{decstab} requires  that the value of  $h_n$ at any $\wh x\in \wh{\mathbb X}$ depends only on the points of the process that are in the ball of a  radius $R_n$  (Note  that the original definition of stabilization in \cite{lachieze2019normal} is for stochastic $R_n$, and the results there need to impose an exponential tail bound for its distribution. Our setup provides a deterministic stabilization radius, which trivially satisfies the tail bound, leading both to the simpler \refdef{decstab} and so the simpler Theorem 
\ref{cltthm}.)



\begin{define}[$(4+p)${th} moment condition]
The sequence of functions  $(h_n)_{n\geq 1}$ satisfies a $(4+p)${th} moment condition for some $p\in (0,\infty)$, if there is a constant $c\in(0,\infty)$  such that, for all $\mc A \subset \Rd$ with $|\mc A|\leq 7$,
\begin{equation}\label{eqmom}
    \sup_{n\in[1,\infty)}\sup_{x\in\Rd}\ex\Big| h_n\big( (x,M_x),\ \wh{\mc P}_n\cup \{(x,M_x)\} \bigcup_{y\in \mc A} \{(y, M_{y})\} \big) \Big|^{4+p} \leq c.
\end{equation}
\end{define}



The CLT in \cite{lachieze2019normal} deals with functionals of $\wh{\mcP}_n$ that can be written as follows 
\[
    H_n \coloneqq \sum_{\wh x\in \wh{\mcP}_n} h_n(\wh x,\wh{\mcP}_n).
\]
The following theorem quantifies the rate of  convergence in terms of the \emph{Kolmogorov} distance
 between two random variables $Y$ and $Z$,
\begin{equation*}
    d_K(Y,Z)\coloneqq \sup_{u\in \mathbb R}\big|\pr[Y\leq u] - \pr[Z\leq u]\big|.
\end{equation*}


\begin{thm}[Corollary 2.2 in \cite{lachieze2019normal}]
\label{cltthm}
Assume that $(h_n)_{n\geq 1}$ are stabilized with a decaying radius 
as in Definition \ref{decstab},
and satisfy the $(4+p)${th} moment condition \eqref{eqmom} for some $p>0$. Assume further that there is a constant $C>0$ such that $\sup\limits_{n\geq 1} n/\var(H_n)\leq C$. Then there is a $C'>0$ such that
\[
    d_K\left(\frac{H_n-\ex H_n}{\sqrt{\var H_n}}  ,Z\right)\leq \frac{C'}{\sqrt{\var H_n}}\quad \textrm{for $n\geq 1$,}
\]
where $Z\sim N(0,1)$.

\end{thm}

Next, we use our birth-death Markov process $\eta_n(t)$ to construct a (static) marked Poisson process for which Theorem \ref{cltthm} will apply.
Start by fixing $T>0$, and define
\[
    \eta_n([0,T]) := \bigcup_{t\in[0,T]} \eta_n(t).
\]
From the definition of $\eta_n(t)$ we have that $\eta_n([0,T])$ is a Poisson process on $\Rd$ with intensity $ n (1+ T) \mathbb{Q}$.
Next, for each $x\in \eta_n([0,T])$,  set
\begin{equation}\label{eqn:def_B_L}
\begin{split}
    B_x &:= \inf\{t\ge 0 : x\in \eta_n(t)\},\\
    D_x &:= \sup\{t\ge 0 : x\in \eta_n(t)\},\\
    L_x &:= D_x-B_x.
\end{split}
\end{equation}
In other words, $B_x$ and $D_x$ are the birth and death times of $x$, respectively, and $L_x$ is its lifetime.
Note that $B_x\in [0,T]$ and $D_x\ge B_x$ (but is not necessarily in $[0,T]$).
The process $\eta_n(0)$ can be retrieved as a thinning of $\eta_n([0,T])$, and therefore
\[
    \pr[B_x = 0] = \frac{1}{1+ T}.
\]
 Since the arrival times of new points is a homogeneous Poisson process (in $t$), then conditioning on $B_x>0$, the birth time is uniformly distributed in $[0,T]$. Combining the last two observations, we can write
\[
    B_x = Z_xY_x,    
\]
where $Z_x,Y_x$ are independent, 
\[\pr[Z_x=0] = 1-\pr[Z_x = 1] = \frac{1}{1+ T},\]
and $Y_x\sim U[0,T]$.
In addition, note that the definition of $\eta_n$ is such that $L_x\hspace{-0.2pt}\sim\hspace{-0.2pt} \mathrm{Exponential(1)}$. 
Finally, given $\eta_n([0,T])$, the values $(B_x,L_x)$ are independent between different points.
To conclude, defining
\begin{equation}\label{eqn:marked_eta}
    \wh{\eta}_{n,T} := \{(x,M_x) : x\in \eta_n([0,T]), M_x = (B_x,L_x)\}, 
\end{equation}
we have that $\wh{\eta}_{n,T}$ is marked Poisson process with independent marks, and so meets the basic structure required for Theorem \ref{cltthm}.
Note  that for all $t\in [0,T]$, we can retrieve the point process $\eta_n(t)$ from $\wh{\eta}_{n,T}$ by taking
\[
\eta_n(t) = \{\wh{x} \in \wh\eta_{n,T} : \tau_t(\wh{x}) =1 \},
\]
where for $\wh{x} = (x,(B_x,L_x))$ we define
\begin{equation}\label{eqn:def_tau}
    \tau_t(\wh{x}) := \ind\{B_x \le t < B_x+L_x\}.
\end{equation}

Therefore, for any $T>t$, we have
\begin{equation}\label{eqn:f_n_tau}
    f_{n}(t) = \sum_{\mcS \subset \eta_n(t)} \xi_{k,r} (\mcS) = \sum_{\wh{\mcS}\subset \wh\eta_{n,T}} \xi_{k,r}(\wh{\mcS})  \tau_t(\wh{\mcS}),
\end{equation}
where $\tau_t({\wh{\mcS}}) := \prod_{\wh{x}\in\wh{\mcS}}\tau_t({\wh{x}})$, and abusing notation we assume that $\xi_{k,r}(\wh{\mcS})$ simply ignores the marks.

We are now ready to prove the convergence of the finite dimensional distributions.
\begin{proof}[Proof of Lemma \ref{finite}] 
Using the Cram\'er-Wold  Theorem  \cite[Theorem 29.4]{MR1324786}, it  suffices to show that
\[
    \frac{\sum_{i=1}^m \omega_i \bar{f}_{n}(t_i)}{\sqrt{\var \sum_{i=1}^m \omega_i \bar{f}_{n}(t_i) }} \xrightarrow{d}N(0,1),
\]
for any fixed $m\in \mathbb N$, $t_1,\ldots,t_m\geq 0$ and $\bs{\omega} = (\omega_1,\ldots,\omega_m)^\top \in \mathbb R^m\setminus \{\bs 0\}$,  where $\xrightarrow{d}$ denotes convergence in distribution. 

For the rest of the proof, fix $m$, $t_1<\ldots< t_m$, $\bs{\omega}$, and $k$. Next, taking the mark space to be  $\mathbb{M} = \mathbb{R}^2$ and using \eqref{eqn:marked_eta} and \eqref{eqn:def_tau},  set
\[
     h_n(\wh x,\wh{\eta}) \coloneqq 
     \frac{1}{k}\sum_{\substack{\wh{\mc S} \subset \wh{\eta}\\ \wh x\in \wh{\mcS}  } } \xi_{k,r}(\wh{\mcS})
     \sum_{i=1}^m {\omega_i} \tau_{t_i}(\wh{\mcS}).
\]
From \eqref{eqn:f_n_tau} we then have that
\begin{align*}
H_n &:= \sum_{\wh{x}\in \wh{\eta}_{n,T}} h_n(\wh{x}, \wh{\eta}_{n,T}) = \sum_{i=1}^m \omega_i f_n(t_i). 
\end{align*}
\efe{In addition, due to stationarity, we have
\begin{align*}
\frac{\sum_{i=1}^m \omega_i \bar{f}_{n}(t_i)}{\sqrt{\var \sum_{i=1}^m \omega_i \bar{f}_{n}(t_i) }} &= \frac{\sum_{i=1}^m \omega_i f_n(t_i) - \ex\left[\sum_{i=1}^m \omega_i f_n(t_i)\right]}{\sqrt{\var[f_n(0)]}\sqrt{\var \sum_{i=1}^m \omega_i \frac{f_{n}(t_i)}{\sqrt{\var[f_n(0)]}} } } = \frac{H_n-\ex[H_n]}{\sqrt{\var[H_n]}}.
\end{align*}}
Thus, in order to complete the proof we only need to verify that the conditions of Theorem \ref{cltthm} hold for every such $H_n$.

We start with the issue of locality.
Since $\xi_{k,r}(\mcS)$ is local due to \refass{localass}, the functions $h_n$ are  stabilized with  decaying radius   $R_n =\rmax  r$, where $\rmax$ is defined in \refass{localass}. Thus, under the conditions of the lemma,
$\limsup n(R_n)^d <\infty$, as required in Definition \ref{decstab}.
For the moment condition \eqref{eqmom}, note that $h_n$ can be bounded as follows
\[
    h_n(\wh x,\wh{\eta})\leq  c_{\bs\omega}
     \sum_{\substack{\wh{\mc S} \subset \wh{\eta}\cap \wh{B}(x,\rmax r) \\ \wh x\in \wh{\mcS}   } } \xi_{k,r}(\wh{\mcS}),
\]
where $\wh{B}(x,r) = B(x,r)\times \mathbb{R}^2$, and where $c_{\bs\omega}$ is a constant that depends on $\bs\omega$ and $k$. Therefore,  we have
 \[
      \ex\Big[\Big| h_n\big( \wh{x}, \wh{\eta}_{n,T}\cup \{\wh{x}\} \bigcup_{y\in \mc A} \{(y, M_{y})\} \big) \Big|^{5}\Big]
     \leq \big(c_{\bs\omega}\|\xi_{k}\|_{\infty}\big)^{5}   \ex\left[ \binom{7+\big|\wh{\eta}_{n,T}\cap \wh{B}(x,\rmax r)\big|}{k}^5\right],
 \]
  for any fixed $\mc{A}$ with $|\mc{A}|\le 7$, and any $\wh{x}\in\Rd\times \mathbb{R}^2$.  
 
 Writing $Z := \big|\wh{\eta}_{n,T}\cap \wh{B}(x,\rmax r)\big|$,  recall that $Z\sim\mathrm{Poisson}(\mu)$, with \[\mu=n(1+ T)\mathbb{Q}\big[B(x,\rmax r)\big].\]
 
 Since $nr^d$ is bounded, we have that $\mu$ is uniformly bounded in $\Rd$. In addition, note that $\binom{7+Z}{k}^5$ is a polynomial in $Z$ of degree $5k$. As the moments of $Z$ are themselves polynomials in $\mu$, we obtain a uniform upper bound for the right hand side above, and thus the $(4+p)$th moment condition \eqref{eqmom} holds for $p=1$.
 
 Finally, to check the last condition in Theorem \ref{cltthm}, we will use stationarity and the limiting covariance function given in \reflem{covgen}, which will be proven in the next section,
\[
     \lim_{n\to\infty} \frac{\var(H_n) }{\var\left(f_{n}(0)\right)}  =\sum_{i,j=1}^m \omega_i \omega_j\sum_{\ell=1}^k \lambda_\ell e^{-\ell|t_i-t_j|} 
     =  \sum_{\ell=1}^k \lambda_\ell\bs{\omega}^\top \bs T^{(\ell)} \bs{\omega},\nn
\]
for a set of positive constants $\lambda_1,\ldots,\lambda_k$, with the entries of the matrix $\bs{T}^{(\ell)}$ defined as 
 \[T^{(\ell)}_{ij} = e^{-\ell|t_i-t_j|}. \]

Since $e^{-\ell t}$ is a covariance function (of an OU process) it is immediate that 
$\bs{T}^{(\ell)}$ is a positive definite matrix.
Therefore, using \eqref{eqvarbound}, we have 
\[\frac{n}{\var(H_n)}\leq \frac{n }{ c \var(f_{n}(0))  }      \leq  \frac{c'}{(nr^d)^{k-1}},
\]
for constants $c, c'>0$.
Consequently, $\var(H_n) \to \infty$, which, together with the fact that we have established that all the conditions of Theorem \ref{cltthm} hold, completes the proof of the lemma.

\end{proof}




\subsection{The Limiting Covariance Function}

\begin{lemma}\label{covgen}
 If $nr^d\to\gamma\in(0,\infty)$, for any $t_1,t_2\geq 0$, the covariance of the normalized additive statistic $\bar f_{n}(t)$ satisfies
  \begin{align*}
    \lim_{n\to\infty}\cov[\bar{f}_{n}(t_1),\bar{f}_{n}(t_2)] = 
	     \sum\limits_{j=1}^{k}\lambda_{j} e^{-j  |t_1-t_2| } ,
\end{align*}
 where the $\lambda_{j}$ are positive constants (depending on $\gamma$, $q$, and $\xi_{k,1}$) satisfying  $\sum\limits_{j=1}^{k}\lambda_{j} = 1$.
\end{lemma}

\begin{proof}[Proof of \reflem{covgen}]
Suppose that $t_1\leq t_2$, and let $\Delta = t_2-t_1$. Stationarity implies that it suffices to compute $\cov[f_n(0),f_n(\Delta)]$ for all $\Delta\geq 0$. Fix $T>\Delta$, and recall the definition of the marked Poisson process $\wh{\eta}_{n,T}$ in \eqref{eqn:marked_eta}.
Then,
\begin{align*}
    \ex[&f_{n}(0)f_{n}(\Delta)] = \ex\Big[\sum_{\wh{\mc S}_1, \wh{\mc S}_2 \subset \wh{\eta}_{n,T}} \xi_{k,r}(\wh{\mcS}_1) \xi_{k,r}(\wh{\mcS}_2) \tau_0(\wh{\mcS}_1)\tau_\Delta(\wh{\mcS}_2)\Big]
\end{align*}
Using the generalized Mecke's formula \eqref{Palmhigh} and the independence of marks, as for the case in  \eqref{126} we have
\begin{align}
    \ex[f_{n}(0)f_{n}(\Delta)] = \sum_{j=0}^{k}\frac{[n(1+ T)]^{2k-j}}{j!\big((k-j)!\big)^2}  \alpha_{k,j} p_0^{k-j} p_\Delta^{k-j} p_{0,\Delta}^j,
    \label{corr}
\end{align}
where the $\alpha_{k,j}$ are defined in \eqref{probdef}. Let $(B,L)$ denote the mark of an arbitrary point, i.e., they have the same distribution as $(B_x,L_x)$ defined in \eqref{eqn:def_B_L}. Then,
\begin{equation}\label{pdelta}
\begin{split}
    p_0 &= \pr[B=0] = \frac{1}{1+ T},   \\
    p_\Delta &= \pr[B\le \Delta < B+L] = \frac{1}{1+ T}e^{-\Delta} + \frac{ T}{1+ T} \int_0^\Delta \frac{1}{T} e^{-(\Delta-b)}db = \frac{1}{1+ T}, \\
    p_{0,\Delta} &= \pr[B=0, \Delta < L] = \frac{e^{-\Delta}}{1+ T}. 
\end{split}
\end{equation}
Thus, we have
\begin{align*}
    \ex[f_{n}(0)f_{n}(\Delta)] 
     = n^{2k}\sum_{j=0}^{k}\frac{\left(n^{-1}e^{- \Delta }\right)^j  \alpha_{k,j} }{j!\big((k-j)!\big)^2},
\end{align*}
and since $\ex[f_n(0)]=\ex[f_n(\Delta)] =\frac{n^k}{k!}\alpha_k$, and $\alpha_k^2 = \alpha_{k,0}$, we have 
\[
\cov[f_n(0),f_n(\Delta)] =  n^{2k}\sum_{j=1}^{k}\frac{\left(n^{-1}e^{- \Delta }\right)^j  \alpha_{k,j} }{j!\big((k-j)!\big)^2}.
\]
With this we obtain,
\begin{align}\label{eqn:cov_form}
    \cov[\bar{f}_{n}(t_1), \bar{f}_{n}(t_2)] &=
    \frac{\cov[f_n(0)f_n(\Delta)]}{\var[f_{n}(0)]}
     = \frac{ \sum_{j=1}^{k}   \frac{\left(n^{-1}e^{- \Delta }\right)^j \alpha_{k,j} }{j!((k-j)!)^2} }{\sum_{j=1}^{k}   \frac{n^{-j} \alpha_{k,j} }{j!((k-j)!)^2}}.
\end{align}

Note that from the definition of $\alpha_{k,j}$ \eqref{probdef}, and applying \reflem{lemgeom}, we have that, for every $j\geq 1$, there exists $\kappa_{k,j}\in \mathbb{R}^+$ such that,
\[
\lim_{n\to\infty}n^{-j}\alpha_{k,j}= \lim_{n\to\infty} (nr^d)^{-j} r^{d(2k-1)} \kappa_{k,j}, 
\]
where 
\begin{align*}
    \kappa_{k,j}\coloneqq \int_{\Rd}[q(x)]^{2k-j}\der x \int_{(\Rd)^{2k-j-1}} \xi_{k,1}(0, \mb y) \xi_{k,1}(\mb y') \prod_{i=1}^{2k-j-1} \der y_i,
\end{align*}
and where $(0,\mb y) = (0,y_1,\ldots,y_{k-1})$ and $\mb y' = (y_{k-j},\ldots, y_{2k-j-1})$.
Using $nr^d \to \gamma \in (0,\infty)$, if we define
\[
\tilde \lambda_{j} =\frac{\gamma^{-j}\kappa_{k,j}}{j!\big((k-j)!\big)^2},
\]
and $\lambda_{j} = \tilde{\lambda}_{j} / \sum_i\tilde{\lambda}_{i}$ then
\[
\lim_{n\to\infty}\cov[\bar{f}_{n}(t_1), \bar{f}_{n}(t_2)] = \sum_{j=1}^k \lambda_{j}e^{-j\Delta}.
\]
 Furthermore, for all $1\leq j\leq k$ we have that $\lambda_j> 0$, since using the properties of $\xi_{k,1}$ we have
\[    \kappa_{k,j}  
    = \int_{\Rd}[q(x)]^{2k-j}\der x \int_{(\Rd)^{j}} \Xi_{k,1}( \mb y'') \int_{\Rd} \Xi_{k,1}(\mb y''-z )  \der z\, \prod_{i=1}^{j} \der y_i >  0,  
\] where
\begin{align*}
    \Xi_{k,1}( \mb y'')\coloneqq \int_{(\Rd)^{k-j-1}} \xi_{k,1}( 0,\mb y)\prod_{i=j+1}^{k-1} \der y_i
\end{align*}
and $\mb y'' = (y_1,\ldots,y_j)$. The strict inequality follows from 
\begin{align*}
    \int_{(\Rd)^{j}}\Xi_{k,1}( \mb y'') \prod_{i=1}^{j} \der y_i = \int_{(\Rd)^{k-1}}\xi_{k,1}(0,\mb y) \prod_{i=1}^{k-1} \der y_i >0, 
\end{align*}
which is a consequence of  \eqref{feaseq0}. 
This completes the proof.
\end{proof}

\begin{remark}
Although our focus in this paper is on the thermodynamic regime,  note that the proof of \reflem{covgen} also works for the dense ($nr^d\to\infty$) and the sparse ($nr^d\to 0$) regimes. If $nr^d\to\infty$, then the dominant term in \eqref{eqn:cov_form} is for $j=1$. In contrast, if $nr^d \to 0$, then the dominant term is for $j=k$. It therefore follows that, for all $k$,
  \begin{align*}
    \lim_{n\to\infty}\cov[\bar{f}_{n}(t_1),\bar{f}_{n}(t_2)] = \begin{cases}
	    e^{-|t_1-t_2|} & \text{if}\quad n{r^d}\to \infty,\\
	    e^{-k|t_1-t_2|} & \text{if}\quad nr^d\to 0.
	\end{cases}
\end{align*}
\end{remark}

\subsection{Tightness}

\begin{lemma}\label{tight}
    If $nr^d\to\gamma\in(0,\infty)$, the sequence $\{\bar{f}_{n}\}_{n\ge 1}$ is tight.
\end{lemma}

\begin{proof}
We check the conditions for the tightness given in \refthm{mainthm2}. Condition (C1) holds trivially for $\Upsilon=2$ by noting that
\begin{align*}
    \ex[\bar{f}_{n}(\delta) -\bar{f}_{n}(0) ]^2 = 2- 2\cov[\bar{f}_{n}(\delta), \bar{f}_{n}(0) ] ,
\end{align*}
and that
\begin{align*}
    \lim_{\delta\to 0} \lim_{n\to\infty} \cov[\bar{f}_{n}(\delta) , \bar{f}_{n}(0) ] =1.
\end{align*}

We will show that (C2) holds for $\Upsilon_1=\Upsilon_2 = 2$. Due to stationarity, (C2) holds if there exists a constant $C_f$ such that
\begin{align*}
     \ex[(f_{n}(2h)-f_{n}(h))^{2} (f_{n}(h)-f_{n}(0))^{2} ]&\leq C_f h^2 \left(\var[f_{n}(0)]\right)^2.
\end{align*}
From to \eqref{eqvarbound}, we have that
$\var[f_{n}(0)] \geq \kappa_k n$.
Thus, it is enough to show that
\begin{align}
     \ex[(f_{n}(2h)-f_{n}(h))^{2} (f_{n}(h)-f_{n}(0))^{2} ]&\leq C h^2 n^2, \label{450}
\end{align}
for some constant $C>0$.
We would like to employ our marked process representation $\wh{\eta}_{n,T}$ \eqref{eqn:marked_eta}.
Recall the definition of $\tau_t(\wh{x})$ \eqref{eqn:def_tau}, and $\tau_t(\wh{\mcS})$. For $t_1 < t_2$, we also define
\begin{align*}
\tau_{t_1,t_2}(\wh{x}) &\coloneqq \tau_{t_2}(\wh{x}) - \tau_{t_1}(\wh{x}),\\ 
    \tau_{t_1,t_2}(\wh{\mcS}) &\coloneqq \tau_{t_2}(\wh{\mcS}) - \tau_{t_1}(\wh{\mcS}).
\end{align*}
Fix any $T > T_0+1$, then for every $t,h>0$ we have
\[
    f_n(t+h) - f_n(t) = \sum_{\wh{\mcS}\subset \wh\eta_{n,T}} \xi_{k,r}(\wh{\mcS})  \tau_{t,t+h}(\wh{\mcS}).
\]
Thus, we can write
\begin{align}
     (f_{n}(2h)-f_{n}(h))^{2} (f_{n}(h)-f_{n}(0))^{2} =\sum_{\bar{\mcS} \in \left(\fin({\wh{\eta}_{n,T})}\right)^4} g_h(\bar{\mcS}) \label{tig},
\end{align}
where $\bar{\mcS} = (\mcS_1,\mcS_2,\mcS_3,\mcS_4)$ ranges over all possible quadruplets of finite subsets of $\wh{\eta}_{n,T}$, and
\begin{equation}
    \label{eqn:def_g}
    g_h(\bar{\mcS}) :=  \tau_{0,h}(\mcS_1)\tau_{0,h}(\mcS_2)\tau_{h,2h}(\mcS_3)\tau_{h,2h}(\mcS_4)\prod_{i=1}^4 \xi_{k,r}(\mcS_i).
\end{equation}
Note that $g_h$ is nonzero only if $|\mcS_i| = k$ for all $i$. However, the intersections $\mcS_i\cap \mcS_j$ can be of any size (up to $k$). Recalling Definition \ref{defn:pattern}, in order to explore all quadruplets $\bar{\mcS}$ we can concentrate on all possible intersection patterns $\cI = \cI_{\ell}$, with $\ell=4$.
We can thus write the right hand side of \eqref{tig} as
\[
\sum_{\bar{\mcS} \in \left(\fin({\wh{\eta}_{n,T})}\right)^4} g_h(\bar{\mcS})  = 
    \sum_{\cI} \sum_{\bar{\mcS} \in \left(\fin({\wh{\eta}_{n,T})}\right)^4} \ind\{\bar{\mcS}\in \cI\}  g_h( \bar{\mcS}). 
\]

In order to evaluate the expectation of the sum above, we need some more notation. Fix an intersection pattern $\cI$, and let $\wh{\mb X}\coloneqq (\wh{X}_1,\ldots,\wh{X}_{|\cI|})$ be iid marked points in $\Rd\times \mathbb{R}^2$. 
Next, set 
\[\bar\mcT = (\mcT_1,\mcT_2,\mcT_3,\mcT_4) =  \Pi_{\cI}(\wh{\mb X}),\]
so that $\bar\mcT$ is actually a  splitting of the variables according to $\cI$.
Using \reflem{Palmhigh} we have
\begin{equation}\label{eqn:E_I}
E_{\cI} := \ex\left[\sum_{\bar{\mcS} \in \left(\fin({\wh{\eta}_{n,T})}\right)^4} \ind\{\bar{\mcS}\in \cI\}  g_h( \bar{\mcS})\right] = \frac{\big(n(1+ T)\big)^{|\cI|}}{\prod_{J}I_{J}!} \ex[g_h(\bar{\mcT})].
\end{equation}
From Lemma \ref{lem:Phi} in Appendix \ref{exph}, we have that 
\[
\ex[g_{h}(\bar{\mcT})] = \frac{1}{(1+ T)^{|\cI|}}\Phi_{\cI}(h)\ex\Big[\prod_{i=1}^4 \xi_{k,r}(\mcT_i)\Big], 
\]
where $\Phi_{\cI}(h)$ is a continuous function defined in \eqref{eqn:def_Phi}. Using \eqref{PhicI2} and 
Taylor's theorem with the Lagrange form of remainder, for any given $h>0$, there exists $h^*\in [0,h]$ such that
\begin{align}\label{PhicI}
|\Phi_{\cI}(h)| = \left|h^2\frac{\Phi_{\cI}''(h^*)}{2}\right| \le C_{\cI} h^2,
\end{align}
for some $C_{\cI}>0$.
Finally, from \reflem{lemgeom}, we have that
\begin{align*}
    \ex\Big[\prod_{i=1}^4 \xi_{k,r}(\mcT_i)\Big] \le C'_{\cI} r^{d(|\cI| - \#_{\cI})}.
\end{align*}
Putting everything back into \eqref{eqn:E_I}, we obtain
\[
|E_{\cI}| \le C_{\cI}C_{\cI}' h^2  \frac{n^{|\cI|}}{\prod_{J}I_{J}!} r^{d(|\cI| - \#_{\cI})} = C_{\cI}'' h^2 n^{\#_{\cI}}(nr^d)^{|\cI|-\#_{\cI}}.
\]
Finally, note that if $\cI$ is such that  there exists $i$ where $\mcT_i \cap \mcT_j = \emptyset$ for all $j\ne i$, then we have from \eqref{eqn:def_g} that 
$\ex[g_h(\bar\mcT)]= 0$.
Therefore, we can assume that $\#_{\cI}$ is either $1$ or $2$. In addition, since we are in the thermodynamic limit, we have $nr^d\to\gamma$. Thus, we conclude that
\[
|E_{\cI}| \le C_{\cI}'''h^2 n^{2}.
\]
Since the collection of all possible patterns $\cI$ is finite and independent of $n$, we have established \eqref{450}, concluding the proof.

\end{proof}

\begin{proof}[Proof of \refthm{OU}]
 The statement of the theorem is a straightforward consequence of Lemmas \ref{finite}, \ref{covgen}, \ref{tight}, and \refthm{mainthm}.
\end{proof}

\section{Proof of \refthm{OU2}}\label{proof2}\sec
While the general structure of the proof of \refthm{OU2} is virtually the same as that of \refthm{OU} in the previous subsection, some of the details are quite different, and so we will take a slightly different route through the various steps.

\subsection{Moments and the Limiting Covariance Function}

\begin{lemma}\label{covgen2}
If $nr^d\to\gamma\in(0,\infty)$, then for any $t_1,t_2\geq 0$, the covariance of the normalized additive statistic $\bar F_{n}(t)$ satisfies
\[
\lim_{n\to\infty}\cov[\bar F_{n}(t_1),\bar F_n(t_2)] = \sum_{l=1}^\infty \Lambda_l e^{-l |t_1-t_2| } 
\]
for some $\Lambda_l>0$ (depending on $\gamma$, $q$,  $\xi_{k,1}$, and $N_1$) with $\sum_{l=1}^\infty  \Lambda_l=1$.

\end{lemma}
 
 \begin{proof}
 We start with first moments. Observe that 
  \begin{align}
    F_{n}(t) = \sum_{\wh{\mcS}\subset \wh\eta_{n,T}} \xi_{k,r}(\wh{\mcS})  \tau_t(\wh{\mcS})\tau^*_t\big( (\wh{\eta}_{n,T}\setminus \wh{\mcS})\cap N_r(\wh\mcS)\big),  \label{Fnt2}
  \end{align}
 where $\wh{\eta}_{n,T}$ is the same marked process we defined previously, and 
  \begin{align*}
     \tau^*_t({\wh{\mc{A}}}) \coloneqq \prod_{\wh{x}\in\wh{\mc{A}}}[1-\tau_t({\wh{x}})]
 \end{align*}
 for some $\wh{\mc{A}}\subset\Rd$.
  Using stationarity and  Mecke's formula (\reflem{Palmhigh}),
\begin{align*}
    \ex[F_n(t)] = \ex[F_n(0)]=\frac{n^k}{k!}\ex\left[\xi_{k,r}(\mb{X}) \tau^*_t\big(\wh{\mcP}_{n}\cap N_r(\mb{X}) \big)\right], 
\end{align*}
where $\mb{X}$ is a $k$-tuple of iid points with density $q$, and $\wh{\mcP}_{n}$ is a marked Poisson point process with density $nq$, independent of $\mb{X}$.
Note that using the void probabilities for Poisson processes,
\begin{align*}
    \ex\left[\tau^*_t\big(\wh{\mcP}_{n}\cap N_r(\mb{X}) \big)\Big|\mb{X} \right] = e^{-n\bbQ[N_r(\mb{X})]}.
\end{align*}
Therefore,
\[
    \ex[F_n(t)] = \frac{n^k}{k!}\ex\left[\xi_{k,r}(\mb{X}) e^{-n\bbQ[N_r(\mb{X})]}\right].
\]

For the second moment, we use again Mecke's formula. From \eqref{Fnt2}, similarly to \eqref{corr}, we can write
\begin{align}
    \ex[F_{n}(0)F_{n}(\Delta)] = \sum_{j=0}^{k}\frac{[n(1+ T)]^{2k-j}}{j!((k-j)!)^2}  \alpha_{k,j}^\Delta p_0^{k-j} p_\Delta^{k-j} p_{0,\Delta}^j,\label{Ncorr}
\end{align}
where 
\[
    \alpha_{k,j}^\Delta\coloneqq \ex\Big[\xi_{k,r}(\mb{X}) \xi_{k,r}(\mb{X}') \tau^*_0\big(\wh{\mcP}_{n(1+ T)}\cap N_r(\mb{X}) \big) \tau^*_{\Delta}\big(\wh{\mcP}_{n(1+ T)}\cap N_r(\mb{X}') \big) \Big],
\]
and $\mb X'\coloneqq (X_{k-j+1},\ldots, X_{2k-j})$ are iid points with density $q$.
Note that,
\begin{align}
\begin{split}\label{tau*0}
\tau^*_0\big(\wh{\mcP}_{n(1+ T)}&\cap N_r(\mb{X}) \big)    \tau^*_\Delta\big(\wh{\mcP}_{n(1+ T)}\cap N_r(\mb{X}') \big)\\ 
& = \prod_{\substack{x\in\wh{\mcP}_{n(1+ T)} \\ x\in N_r(\mb{X})\setminus N_r(\mb{X}')}} \ind\{B_x>0\}
\prod_{\substack{x\in\wh{\mcP}_{n(1+ T)} \\ x\in  N_r(\mb{X}')\setminus N_r(\mb{X}) }} \ind\{B_x>\Delta\;\mathrm{or}\; B_x+L_x\leq \Delta\}
     \\&\qquad\qquad\times\prod_{\substack{x\in\wh{\mcP}_{n(1+ T)} \\ x\in N_r(\mb{X})\cap N_r(\mb{X}')}} \ind\{B_x>0\}\ind\{B_x>\Delta\;\mathrm{or}\; B_x+L_x\leq \Delta\}.
     \end{split}
\end{align}
Note that using the independence of birth and lifetimes of different particles, the conditional expectation of the first term above, given $\mb X$ and $\mb X'$, can be written as
\begin{align}
    \sum_{\ell=0}^{\infty}\left(1-\pr[B=0]\right)^\ell \pr\left[\left|\wh{\mcP}_{n(1+T)}\cap \left(N_r(\mb{X})\setminus N_r(\mb{X}')\right)\right| = \ell\right], \label{sumell0}  
\end{align}
which is in the form of the moment generating function of $\left|\wh{\mcP}_{n(1+T)}\cap \left(N_r(\mb{X})\setminus N_r(\mb{X}')\right)\right|$, which is Poisson distributed with mean $n(1+T)\mathbb Q\left[N_r(\mb{X})\setminus N_r(\mb{X}')\right]$. Therefore, using \eqref{pdelta} for $p_0=\pr[B=0]$, the sum in \eqref{sumell0} evaluates to  
\begin{align*}
    e^{-p_0 n(1+T)\mathbb Q\left[N_r(\mb{X})\setminus N_r(\mb{X}')\right]}  = e^{-n\mathbb Q\left[N_r(\mb{X})\setminus N_r(\mb{X}')\right]}.
\end{align*}
Through the same arguments, and the following
\begin{align*}
    \pr[B_x>\Delta\;\mathrm{or}\; B_x+L_x\leq \Delta] = 1-p_{\Delta},
\end{align*}
we evaluate the conditional expectation of the second term of \eqref{tau*0} as $e^{-n\mathbb Q\left[N_r(\mb{X}')\setminus N_r(\mb{X})\right]}$. Lastly, we use that for any $x$,
\begin{align*}
    \pr\left[\ind\{B_x>0\} = \ind\{B_x>\Delta\;\mathrm{or}\; B_x+L_x\leq \Delta\} = 1\right] &=  1- \pr[ B_x = 0 \;\mathrm{or}\; B_x \leq \Delta < B_x + L_x] \\
    &\hspace{-0.1em}=1-(p_0 + p_\Delta - p_{0,\Delta}) = \frac{T+e^{-\Delta}-1}{1+T}
\end{align*}
to find that the conditional expectation of the third term is $e^{-n(2-e^{-\Delta})\bbQ\left[N_r(\mb{X}) \cap N_r(\mb{X}')\right]}$. 
Together with the independence of the three products in \eqref{tau*0}, we obtain
\begin{align*}
     &\ex\left[\tau^*_0\big(\wh{\mcP}_{n(1+ T)}\cap N_r(\mb{X}) \big)    \tau^*_\Delta\big(\wh{\mcP}_{n(1+ T)}\cap N_r(\mb{X}') \big)\Big|\mb{X},\mb{X}'\right] \\&\qquad = \exp\left\{-n\left(\bbQ\left[N_r(\mb{X})\right] +\bbQ\left[N_r(\mb{X}')\right] -e^{-\Delta}\bbQ\left[N_r(\mb{X}) \cap N_r(\mb{X}')\right]  \right)\right\}.
\end{align*}
Note also that
\[    \ex[F_n(0)]\ex[F_{n}(\Delta)] =\frac{n^{2k}}{(k!)^2} \alpha_{k,0}^{\infty},
\]
where
\begin{align*}
    \alpha_{k,0}^{\infty}\coloneqq \ex\Big\{\xi_{k,r}(\mb{X}) e^{-n\bbQ\left[N_r(\mb{X})\right]}&
    \xi_{k,r}(\mb{X}') e^{-n\bbQ\left[N_r(\mb{X}')\right]}\Big\}. 
\end{align*}
Combined with \eqref{Ncorr}, we have
\begin{align}
\cov[F_{n}(0),F_{n}(\Delta)] =  n^{2k}\left\{ \frac{\alpha_{k,0}^\Delta- \alpha_{k,0}^{\infty} }{(k!)^2}+\sum_{j=1}^{k}\frac{\left(n^{-1}e^{- \Delta }\right)^j  \alpha_{k,j}^\Delta }{j!((k-j)!)^2}\right\}. \label{coveul2}
\end{align}

To simplify further, we note the following two results, proofs of which will be given in \refapp{alphaproof}.
For all $k\geq 2$, if $nr^d\to\gamma\in(0,\infty)$ then
\begin{equation}\label{alphaasymp}
\begin{split}
    \lim_{n\to\infty}n^{2k} r^{d} (\alpha_{k,0}^\Delta- \alpha_{k,0}^\infty) &=  \gamma^{2k} \sum_{l=1}^\infty e^{-l \Delta }\frac{\gamma^l}{l!} \kappa_{k,0,l},\q\text{and}\\
    \lim_{n\to\infty}n^{2k-j} r^d \alpha_{k,j}^\Delta &=  \gamma^{2k-j}  \sum_{l=0}^\infty e^{-l \Delta }\frac{\gamma^l}{l!} \kappa_{k,j,l},\q \textrm{for all $1\leq j\leq k$,}
    \end{split}
\end{equation}
    for constants $\kappa_{k,j,l}$ given in \eqref{kappakjl}.
Combined with \eqref{coveul2}, we have
\begin{align*}
\lim_{n\to\infty}r^d\cov[F_{n}(0),F_{n}(\Delta)] =&  \gamma^{2k} \sum_{l=0}^\infty \sum_{j=1}^{k}\frac{\kappa_{k,j,l}  \gamma^{l-j}  }{l! j!((k-j)!)^2} e^{-(l+j) \Delta }+ \gamma^{2k} \sum_{l=1}^\infty \frac{\kappa_{k,0,l}  \gamma^{l}  }{l! (k!)^2} e^{-l \Delta }\nn\\
=&  \sum_{l=1}^\infty \tilde\Lambda_l e^{-l \Delta },
\end{align*}
where
\begin{align*}
    \tilde\Lambda_l \coloneqq \sum_{j=0}^{\min(l, k)}\frac{\gamma^{2k+l-2j} \kappa_{k,j,l}    }{ (l-j)! j!((k-j)!)^2}.
\end{align*}
\efe{Since each step in calculating the covariance still holds for $\Delta=0$, we also have
\begin{align*}
    \lim_{n\to\infty}r^d\var[F_{n}(0)] = \sum_{l=1}^\infty \tilde\Lambda_l.
\end{align*}
Thus, we obtain the desired result by setting $\Lambda_{l} = \tilde{\Lambda}_{l} / \sum_i\tilde{\Lambda}_{i}$, and using the fact that due to stationarity,
\[\cov[\bar F_{n}(t_1),\bar F_n(t_2)] = \frac{\cov[ F_{n}(t_1), F_n(t_2)]}{\var[F_n(0)]}.\]
} The observation that $\Lambda_l>0$ follows from 
\begin{align*}
    \tilde\Lambda_l\geq \frac{\gamma^{2k+l} \kappa_{k,0,l}    }{ l! (k!)^2},
\end{align*}
and \eqref{kappak0l}. This also implies that
\begin{align}\label{varF}
    \liminf_{n\to\infty}\frac{\var[F_{n}(0)]}{n} > 0,
\end{align}
which will be required later.
\end{proof}

\subsection{Finite Dimensional Distributions}

\begin{lemma}\label{finiteF}
If $\lim\limits_{n\to\infty}nr^d\to\gamma\in(0,\infty)$ then the finite dimensional distributions of $\bar{F}_{n}(t)$ converge  to multivariate Gaussian .
\end{lemma}
\begin{proof}
We will again apply  \refthm{cltthm}  to prove this lemma, following the same steps as used in the proof of \reflem{finite}, the notation of which we now adopt freely. Note that,
 \begin{align*}
     \sum_{i=1}^m \omega_i F_n(t_i) = \sum_{\wh{x}\in \wh{\eta}_{n,T}} {h}^*_n(\wh{x}, \wh{\eta}_{n,T}), 
 \end{align*}
 with
 \begin{align*}
     {h}^*_n(\wh{x}, \wh{\eta}_{n,T})\coloneqq \frac{1}{k}\sum_{\substack{\wh{\mc S} \subset \wh{\eta}_{n,T} \\ \wh x\in \wh{\mcS}  } } \xi_{k,r}(\wh{\mcS})
     \sum_{i=1}^m {\omega_i} \tau_{t_i}(\wh{\mcS})\tau^*_{t_i}\big( (\wh{\eta}_{n,T}\setminus \wh{\mcS})\cap N_r(\wh\mcS)\big). 
 \end{align*}

Since $N_r(\wh\mcS)\subset B\big(x, \beta_k \max(\diam(\mcS), r)\big)$, due to \refass{loc2}, we have that ${h}^*_n(\wh{x}, \allowbreak\wh{\eta}_{n,T})$ is stabilized with a decaying radius $\beta_k r \max(1, \rmax)$. The bounded $(4+p)${th} moment condition also holds for ${h}^*_n$ since 
\begin{align*}
    |{h}^*_n(\wh{x}, \wh{\eta}_{n,T})|\leq |h_n(\wh{x}, \wh{\eta}_{n,T})|.
\end{align*}
Showing the last condition in the statement of \refthm{cltthm} on the variance of $\bar F_n(t)$, follows the same steps in the proof of 
\reflem{finite}, using \eqref{varF} and the limit covariance given in \reflem{covgen2}. 

\end{proof}

\subsection{Tightness}

\begin{lemma}\label{tightF}
    If $nr^d\to\gamma\in(0,\infty)$, the sequence $\{\bar{F}_{n}\}_{n\ge 1}$ is tight.
\end{lemma}
\begin{proof}
Note that the steps in the proof of  \reflem{tight}, leading to  \eqref{tig}, can be repeated mutatis mutandis for $F_n(t)$. Therefore, condition (C1) of tightness holds for $\Upsilon=2$, and we can write
\begin{align}
     \ex\left[(F_{n}(2h)-F_{n}(h))^{2} (F_{n}(h)-F_{n}(0))^{2}\right] =\ex\left[\sum_{\bar{\mcS} \in \left(\fin({\wh{\eta}_{n,T})}\right)^4} g_{h}(\bar{\mcS}, \wh{\eta}_{n,T})\right] ,\label{F2h}
\end{align}
where
\[
    g_{h}(\bar{\mcS}, \wh{\eta}_{n,T}) \coloneqq
    \tau_{0,h}(\mcS_1,\wh{\eta}_{n,T})\tau_{0,h}(\mcS_2,\wh{\eta}_{n,T})\tau_{h,2h}(\mcS_3,\wh{\eta}_{n,T})\tau_{h,2h}(\mcS_4,\wh{\eta}_{n,T}) \times\prod_{i=1}^4 \xi_{k,r}(\mcS_i),
\]
and 
\[
    \tau_{t,t+h}(\mcS,\wh{\eta}_{n,T})\coloneqq \tau_{t+h}(\mcS)\tau^*_{t+h}\big( (\wh{\eta}_{n,T}\setminus {\mcS})\cap N_r(\mcS)\big) -  \tau_{t}(\mcS)\tau^*_t\big( (\wh{\eta}_{n,T}\setminus {\mcS})\cap N_r(\mcS)\big).
\]
Due to spatial independence, we have that 
$\ex\left[g_{h}(\bar{\mcS}, \wh{\eta}_{n,T})\big|\bar{\mcS}\right] =0$ if there is $i\in\{1,\ldots,4\}$ such that $N_r(\mcS_i)\cap N_r(\mcS_\ell) = \varnothing$ for all $\ell \ne i$. 
Therefore,
\[
    g_{h}(\bar{\mcS}, \wh{\eta}_{n,T}) =
    \tau_{0,h}(\mcS_1,\wh{\eta}_{n,T})\tau_{0,h}(\mcS_2,\wh{\eta}_{n,T})\tau_{h,2h}(\mcS_3,\wh{\eta}_{n,T})\tau_{h,2h}(\mcS_4,\wh{\eta}_{n,T}) \prod_{i=1}^4 \wt\xi_{k,r}(\mcS_i),
\]
where
\begin{align*}
    \wt\xi_{k,r}(\mcS_i) &\coloneqq \xi_{k,r}(\mcS_i) \ind\big\{\exists\, \ell\neq i: N_r(\mcS_i)\cap N_r(\mcS_\ell)\neq \varnothing \big\}\\
    &\leq \xi_{k,r}(\mcS_i) \sum_{\ell\neq i} \ind\{N_r(\mcS_i)\cap N_r(\mcS_\ell)\neq \varnothing \}.
\end{align*}
Observe that 
\begin{align}
    \prod_{i=1}^4 \wt\xi_{k,r}(\mcS_i)\leq \sum_{\pi}\prod_{i=1}^4 \xi_{k,r}(\mcS_i)\zeta_r(\mcS_i, \mcS_{\pi(i)}), \label{whxi}
\end{align}
where the sum is over all mappings $\pi$ from $\{1,2,3,4\}$ to $\{1,2,3,4\}$ with no fixed points, and 
\[
    \zeta_r(\mcS_i, \mcS_{\pi(i)})\coloneqq \ind\hspace{-0.4pt}\left\{N_r(\mcS_i)\cap N_r(\mcS_{\pi(i)})\neq \varnothing \right\} \ind\{\diam(\mcS_i)\leq \rmax r\}
    \ind\{\diam(\mcS_{\pi(i)})\leq \rmax r\}, 
\]
where the conditions in the last two indicators are equivalent to the locality condition (\refass{localass}) on $\xi_{k,r}$. Due to \refass{loc2} on $N_r$, 
\begin{align*}
    \zeta_r(\mcS_i, \mcS_{\pi(i)}) > 0 \quad\textrm{only if}\quad \diam(\mcS_i\cup \mcS_{\pi(i)})\leq 2\beta_k r  \max(\rmax, 1),
\end{align*}
and so $\zeta_r$ is local. 
Going back to the fourth moment calculation, proceeding as in the proof of \reflem{tight},  and using \reflem{Palmhigh}, we find that the right hand side of \eqref{F2h} is equal to
\begin{align*}
    \sum_\cI \frac{(n(1+ T))^{|\cI|}}{\prod_{J}I_{J}!} \ex\big[g_{h}(\bar{\mcT},\wh\mcP_{n(1+ T)})\big], 
\end{align*}
where $\wh\mcP_{n(1+ T)}$ is a marked Poisson point process independent of the iid points in $\bar{\mcT}$ which adhere to the intersection pattern $\cI$. By \reflem{lem:PhiF}, and \eqref{whxi}
\begin{align*}
\left|\ex[g_{h}(\bar{\mcT},\wh\mcP_{n(1+ T)})]\right| \leq \frac{K_{\cI} h^2}{(1+ T)^{|\cI|}}\sum_\pi   \ex\Big[\prod_{i=1}^4 \xi_{k,r}(\mcT_i) \zeta_r(\mcT_i, \mcT_{\pi(i)}) \Big] .
\end{align*}
Translation and scale-invariance of $\zeta_r$ also follows due to the affine invariance (\refass{linF}) of $N_r$. Therefore, we can apply \reflem{lemgeom} to the right hand side, and have
\begin{align*}
    \ex\Big[\prod_{i=1}^4 \xi_{k,r}(\mcT_i) \zeta_r(\mcT_i, \mcT_{\pi(i)}) \Big] r^{-d(|\cI| - \#_{\cI_\pi})}<\infty,
\end{align*}
where $\#_{\cI_\pi}$ denotes the number of components in the intersection structure that \[\prod_{i=1}^4 \xi_{k,r}(\mcT_i) \zeta_r(\mcT_i, \mcT_{\pi(i)})\] constitutes, in which we treat $\mcT_i \cup \mcT_{\pi(i)}$ as a separate node, in addition to $\mcT_i$ and  $ \mcT_{\pi(i)}$, in the intersection graph of the sets.
Therefore $\#_{\cI_\pi}\leq 2$ for all $\cI$ and $\pi$, and
\begin{align*}
     \ex\left[(F_{n}(2h)-F_{n}(h))^{2} (F_{n}(h)-F_{n}(0))^{2}\right] &\leq h^2 \sum_{\cI} \sum_\pi K_{\cI,\pi}n^{\#_{\cI_\pi}}\\
     &\leq K h^2n^2
\end{align*}
for some set of constants $K_{\cI,\pi}$, and $K>0$.
Condition (C2) of tightness for $\bar F_n$ follows from
\begin{align*}
     \ex\left[(\bar F_{n}(2h)-\bar F_{n}(h))^{2}  (\bar F_{n}(h)-\bar F_{n}(0))^{2}\right]=
     \frac{\ex\big[(F_{n}(2h)-F_{n}(h))^{2} (F_{n}(h)-F_{n}(0))^{2}\big]}{\var^2[F_n(0)]} 
\end{align*}
and \eqref{varF}.
\end{proof}

\subsection{Continuity of the Limit Process}

\begin{lemma}\label{contF}
    If $nr^d\to\gamma\in(0,\infty)$, then the limit $\lim\limits_{n\to\infty}\bar{F}_{n}(t)$ is almost surely continuous on $[0,\infty)$.
\end{lemma}
\begin{proof}
We use Fernique's continuity criterion \cite{fernique1964continuite} to prove this, which states that a stationary Gaussian process $X(t)$ is almost surely path continuous if there exists a non-decreasing function $\psi(h)$ such that 
\begin{align}
    \ex\left[(X(h)-X(0))^2\right]\leq \psi^2(h)
\quad \text{and}\quad
    \int_{1}^\infty \psi(e^{-x^2})\der x<\infty. \label{fern}
\end{align}
The reader is referred to \cite{marcus1970continuity} or \cite{garsia1970real} for different proofs of this criterion. 
By Fatou's Lemma and the weak convergence of $\bar F_n$ to the limit process $X$, we have 
\begin{align}
    \ex\left[(X(h)-X(0))^2\right]\leq 
    \lim_{n\to\infty}  \ex\left[ (\bar F_{n}(h)-\bar F_{n}(0))^{2}\right].  
    \label{Fatou}
\end{align}
To compute this limit, 
as in \eqref{F2h}, we can write 
\begin{align*}
     \ex\left[ (F_{n}(h)-F_{n}(0))^{2}\right] =\ex\sum_{\bar{\mcS} \in \left(\fin({\wh{\eta}_{n,T})}\right)^2} \overline g_h(\bar{\mcS}, \wh{\eta}_{n,T}) ,
\end{align*}
where
\begin{align}
    \overline g_{h}(\bar{\mcS}, \wh{\eta}_{n,T}) \coloneqq
    \tau_{0,h}(\mcS_1,\wh{\eta}_{n,T})\tau_{0,h}(\mcS_2,\wh{\eta}_{n,T}) \xi_{k,r}(\mcS_1) \xi_{k,r}(\mcS_2). \label{tilgdef}
\end{align}
 
Following the same arguments as in the proof of \reflem{tightF}, and using \reflem{lem:F2mom} instead of \reflem{lem:PhiF} we obtain
\begin{align*}
     \ex\left[ (F_{n}(h)-F_{n}(0))^{2}\right] 
     \leq K' h n^{|\cI|} r^{d(|\cI| - \#_{\cI})},
\end{align*}
for some constant $K'>0$. However, in this case, $\#_{\cI}=1$, and so, applying \eqref{varF}, we obtain the  bound, uniform in $n$,
\begin{align*}
     \ex\left[ (\bar F_{n}(h)-\bar F_{n}(0))^{2}\right] 
     \leq K h
\end{align*}
for some $K>0$. Thus, choosing $\psi(h) = \sqrt{Kh} $, and recalling \eqref{Fatou}, we have 
\begin{align*}
    \int_{1}^\infty \psi(e^{-x^2})\der x= \int_{1}^\infty\sqrt{K} e^{-x^2/2}\der x\leq \sqrt{\frac{K\pi}{2}},
\end{align*}
giving  that \eqref{fern} is satisfied.
\end{proof}

\begin{proof}[Proof of \refthm{OU2}] 
This follows immediately as a consequence of Lemmas \ref{covgen2}, \ref{finiteF}, \ref{tightF}, \ref{contF} and \refthm{mainthm}.
\end{proof}

\begin{acks}[Acknowledgments]
The authors are grateful to Yogeshwaran Dhandapani for useful discussions and advice, and in particular for pointing us in the direction of highly related literature.
\end{acks}

\begin{funding} \ 
\\
EO was supported in part by the Israel Science Foundation, Grants 2539/17 and 1965/19.\\
OB was supported in part by the Israel Science Foundation, Grant
1965/19.\\
 RJA was supported in part by the Israel Science Foundation, Grant  2539/17.
\end{funding}

\begin{appendix}
\section{Auxiliary Lemmas}\label{exph}
\sec 

Here we state and prove auxiliary lemmas, used in the proofs of main theorems, regarding the expectations of some functions defined in the main text. Recall the definition of $g_h$ in \eqref{eqn:def_g}. We start by proving the following lemma.

\begin{lemma}\label{lem:Phi}
Let $\cI$ be an intersection pattern, and let $\wh{\mb{X}}$ be a $|\cI|$-tuple of iid marked points in $\Rd\times \mathbb{R}^2$. Writing $\bar\mcT =(\mcT_1,\mcT_2,\mcT_3,\mcT_4)= \Pi_\cI(\wh{\mb{X}})$, we have
\[
\ex[g_h(\bar{\mcT})] = \frac{1}{(1+ T)^{|\cI|}}\Phi_\cI(h)\ex\Big[\prod_{i=1}^4 \xi_{k,r}(\mcT_i)\Big], 
\]
where $\Phi_\cI(h)$ is defined in \eqref{eqn:def_Phi}.
\end{lemma}

\begin{proof}
Using  the definition of $g_h$ at \eqref{eqn:def_g}, we have
\begin{equation}
\begin{split}
\ex\big[ g_h(\bar{\mcT}) \big] &= \ex\Big[\prod_{i=1}^4 \xi_{k,r}(\mcT_i)\Big]\ex\big[\tau_{0,h}(\mcT_1)\tau_{0,h}(\mcT_2)\tau_{h,2h}(\mcT_3)\tau_{h,2h}(\mcT_4)\big] \\
 &= \ex\Big[\prod_{i=1}^4 \xi_{k,r}(\mcT_i)\Big](E_1-E_2),
 \end{split}\label{Egkr}
\end{equation}
where
\begin{align}
\begin{split}
    E_1 = \ex \big[ 
 \tau_h(\mcT_1)\tau_h(\mcT_2)\tau_{2h}(\mcT_3)\tau_{2h}(\mcT_4)  &+\tau_0(\mcT_1)\tau_0(\mcT_2)\tau_h(\mcT_3)\tau_h(\mcT_4)\\
  + \tau_0(\mcT_1)\tau_0(\mcT_2)\tau_{2h}(\mcT_3)\tau_{2h}(\mcT_4) &+ \tau_h(\mcT_1)\tau_h(\mcT_2)\tau_h(\mcT_3)\tau_h(\mcT_4)\\
 + \tau_0(\mcT_1)\tau_h(\mcT_2)\tau_h(\mcT_3)\tau_{2h}(\mcT_4)&+ \tau_0(\mcT_1)\tau_h(\mcT_2)\tau_{2h}(\mcT_3)\tau_h(\mcT_4)\\
 + \tau_h(\mcT_1)\tau_0(\mcT_2)\tau_h(\mcT_3)\tau_{2h}(\mcT_4)
 &+ \tau_h(\mcT_1)\tau_0(\mcT_2)\tau_{2h}(\mcT_3)\tau_h(\mcT_4)
     \big],        
\end{split}\label{E1}
\end{align}
\begin{align}\label{E2}
\begin{split}
    E_2 = \ex \big[ 
 \tau_0(\mcT_1)\tau_h(\mcT_2)\tau_{2h}(\mcT_3)\tau_{2h}(\mcT_4)  &+\tau_h(\mcT_1)\tau_0(\mcT_2)\tau_{2h}(\mcT_3)\tau_{2h}(\mcT_4)\\
  + \tau_0(\mcT_1)\tau_h(\mcT_2)\tau_h(\mcT_3)\tau_h(\mcT_4) &+ \tau_h(\mcT_1)\tau_0(\mcT_2)\tau_h(\mcT_3)\tau_h(\mcT_4)\\
 + \tau_0(\mcT_1)\tau_0(\mcT_2)\tau_h(\mcT_3)\tau_{2h}(\mcT_4)&+ \tau_0(\mcT_1)\tau_0(\mcT_2)\tau_{2h}(\mcT_3)\tau_h(\mcT_4)\\
 + \tau_h(\mcT_1)\tau_h(\mcT_2)\tau_h(\mcT_3)\tau_{2h}(\mcT_4)
 &+ \tau_h(\mcT_1)\tau_h(\mcT_2)\tau_{2h}(\mcT_3)\tau_h(\mcT_4)
     \big].
\end{split}
\end{align}
We will explicitly evaluate the first term in $E_1$. The others can be computed in a similar fashion. To start,
 for a set of indexes $J$, write
\begin{align}\label{TJ}
    \mcT_{J} = \left(\bigcap_{j\in J} \mcT_{j}\right)\cap \left(\bigcap_{j\not\in J} (\mcT_{j})^c\right).
\end{align}
The event
$A = \{\tau_h(\mcT_1)\tau_h(\mcT_2)\tau_{2h}(\mcT_3)\tau_{2h}(\mcT_4) =1\}$ can be split as follows
\[
\begin{split}
    A &= \left(\bigcap_{x\in \mcT_1\cup \mcT_2\cup \mcT_{1,2}} \{B_x \le h < B_x+L_x\}\right) \cap \left(\bigcap_{x\in \mcT_3\cup \mcT_4\cup \mcT_{3,4}} \{B_x \le 2h < B_x+L_x\}\right)\\
    &\quad \cap\left(\bigcap_{x \in \mcT_{1,3}\cup \mcT_{1,4}\cup \mcT_{2,3}\cup\mcT_{2,4}\cup\mcT_{1,2,3}\cup\mcT_{1,2,4}\cup\mcT_{1,3,4}\cup\mcT_{2,3,4}\cup\mcT_{1,2,3,4}} \{B_x \le h < 2h < B_x+L_x\}\right).
\end{split}
\]
Note that all events in the above intersections are independent (since the sets of indexes are disjoint), and in addition we can show (cf. \eqref{pdelta}) that 
\[
\begin{split}
    \pr[B_x \le h < B_x+L_x] &= \pr[B_x \le 2h < B_x+L_x] = \frac{1}{1+ T},\\
    \pr[B_x \le h < 2h < B_x+L_x] &= \frac{e^{- h}}{1+ T}.
\end{split}
\]
Therefore, we have
\[
    \pr[A] = \left(\frac{1}{1+ T}\right)^{|\cI|} e^{- h (I_{1,3} + I_{1,4} + I_{2,3} + I_{2,4} + I_{1,2,3} + I_{1,2,4} + I_{1,3,4} + I_{2,3,4}+ I_{1,2,3,4})}.
\]
Applying a similar analysis to all other terms gives
\begin{align}
\begin{split}
    E_1  = \left(\frac{1}{1+ T}\right)^{|\cI|}\Big(&2e^{- h (I_{1,3} + I_{1,4} + I_{2,3} + I_{2,4} + I_{1,2,3} + I_{1,2,4} + I_{1,3,4} + I_{2,3,4}+ I_{1,2,3,4})}\\
    &+ e^{-2 h (I_{1,3}+ I_{1,4} + I_{2,3} + I_{2,4}+ I_{1,2,3}+ I_{1,2,4}+ I_{1,3,4}+ I_{2,3,4} + I_{1,2,3,4} )}+1\\
    &+ e^{- h (I_{1,2}+ I_{1,3} + I_{2,4} + I_{3,4}+ I_{1,2,3}+ I_{2,3,4} )}e^{-2 h (I_{1,4}+ I_{1,2,4} + I_{1,3,4} + I_{1,2,3,4} )}\\
    &+ e^{- h (I_{1,2}+ I_{1,4} + I_{2,3} + I_{3,4}+ I_{1,2,4}+ I_{2,3,4} )}e^{-2 h (I_{1,3}+ I_{1,2,3} + I_{1,3,4} + I_{1,2,3,4} )}\\
    &+ e^{- h (I_{1,2}+ I_{1,4} + I_{2,3} + I_{3,4}+ I_{1,2,3}+ I_{1,3,4} )}e^{-2 h (I_{2,4}+ I_{1,2,4} + I_{2,3,4} + I_{1,2,3,4} )}\\
    &+ e^{- h (I_{1,2}+ I_{1,3} + I_{2,4} + I_{3,4}+ I_{1,2,4}+ I_{1,3,4} )}e^{-2 h (I_{2,3}+ I_{1,2,3} + I_{2,3,4} + I_{1,2,3,4} )}\Big),
\end{split}\label{E1=}
\end{align}
and
\begin{align}
\begin{split}
    E_2  = \left(\frac{1}{1+ T}\right)^{|\cI|}\Big(&e^{- h (I_{1,2} + I_{2,3} + I_{2,4} + I_{2,3,4})}e^{-2 h(I_{1,3} + I_{1,4} + I_{1,2,3} + I_{1,2,4}+ I_{1,3,4} + I_{1,2,3,4})}\\
    &+ e^{- h (I_{1,2}+ I_{1,3} + I_{1,4} + I_{1,3,4})}e^{-2 h (I_{2,3}+ I_{2,4} + I_{1,2,3} + I_{1,2,4} +I_{2,3,4} +I_{1,2,3,4} )}\\
    &+ e^{- h (I_{1,2}+ I_{1,3} +I_{1,4} + I_{1,2,3} + I_{1,2,4} + I_{1,3,4}+  I_{1,2,3,4} )}\\
    &+ e^{- h (I_{1,2}+ I_{2,3} +I_{2,4} + I_{1,2,3} + I_{1,2,4} + I_{2,3,4}+  I_{1,2,3,4} )}\\
    &+ e^{- h (I_{1,3}+ I_{2,3} + I_{3,4} + I_{1,2,3})}e^{-2 h (I_{1,4}+ I_{2,4} + I_{1,2,4} + I_{1,3,4} + I_{2,3,4}+ I_{1,2,3,4} )}\\
    &+ e^{- h (I_{1,4}+ I_{2,4} + I_{3,4}+ I_{1,2,4})}e^{-2 h (I_{1,3}+ I_{2,3}+I_{1,2,3} + I_{1,3,4} + I_{2,3,4} + I_{1,2,3,4} )}\\
    &+  e^{- h (I_{1,4}+ I_{2,4} +I_{3,4} + I_{1,2,4} + I_{1,3,4} + I_{2,3,4}+  I_{1,2,3,4} )}\\
    &+ e^{- h (I_{1,3}+ I_{2,3} +I_{3,4} + I_{1,2,3} + I_{1,3,4} + I_{2,3,4}+  I_{1,2,3,4} )}
    \Big).
\end{split}\label{E2=}
\end{align}
As a shorthand notation to the above, we will write
\begin{equation}\label{eqn:def_Phi}
E_1 - E_2 = \left(\frac{1}{1+ T}\right)^{|\cI|} \Phi_\cI(h),
\end{equation}
concluding the proof.
\end{proof}
\begin{remark}
Using the chain rule, and the symmetry in \eqref{E1=} and \eqref{E2=} (each $I$ term appears in the exponents equally many times in $E_1$ and $E_2$), we observe that
\begin{align}\label{PhicI2}
\Phi_{\cI}(0) = \Phi_{\cI}'(0) = 0, \quad\text{and}\quad \sup_{h\ge 0}|\Phi_{\cI}''(h)| < \infty.
\end{align}
\end{remark}

We now turn to  deriving an upper bound for the expectation of $g_h(\bar{\mcT},\wh\mcP_{n(1+ T)})$, as a function of $h$. 
\begin{lemma}\label{lem:PhiF}
Let $\cI$ and $\bar\mcT$ be as in \reflem{lem:Phi}, then
\[
\left|\ex\big[g_h(\bar{\mcT},\wh\mcP_{n(1+ T)})\big]\right| \leq \frac{K_\cI}{(1+ T)^{|\cI|}}\,   \ex\Big[\prod_{i=1}^4 \wt\xi_{k,r}(\mcT_i) \Big]\, h^2, 
\]
for some finite constant $K_\cI$.
\end{lemma}

\begin{proof}
Following the steps of the previous proof, and exploiting the structural symmetry in the definitions of $g_h(\bar{\mcT})$ and  $g_h(\bar{\mcT},\wh\mcP_{n(1+ T)})$,  observe that
\begin{align}
    \ex\big[g_h(\bar{\mcT},\wh\mcP_{n(1+ T)})\big] = \ex\Big[\prod_{i=1}^4 \wt\xi_{k,r}(\mcT_i)(E^*_1-E^*_2)\Big]\label{Egkr2},
\end{align}
where the definitions of $E^*_1$ and $E^*_2$ are the same as $E_1$ and $E_2$ in \eqref{E1} and \eqref{E2}, except that each $\tau_t(\mcT_j)$ term is to be replaced by 
\begin{align*}
    \tau_t( \mcT_j) \tau^*_t\big( \wh\mcP_{n(1+ T)}\cap N_r(\mcT_j)\big)
\end{align*}
for $t\in\{0,h,2h\}$ and $j\in\{1,2,3,4\}$. Note that the expectation cannot be separated as in \eqref{Egkr}, since the terms $E^*_1$ and $E^*_2$  depend on the location of the points and therefore on $\xi_{k,r}(\mcT_i)$. We start with  the calculation of the first term in $\ex[E^*_1]$,  where the expectation is over the birth and death times of the points in $\bar\mcT$ and the independent Poisson point process $\wh\mcP_{n(1+ T)}$. Note that the event 
\begin{align*}
A^*\coloneqq \big\{ \tau^*_h\big( \wh\mcP_{n(1+ T)}\cap N_r(\mcT_1)\big) &= \tau^*_h\big(\wh\mcP_{n(1+ T)}\cap N_r(\mcT_2)\big) \\& = \tau^*_{2h}\big( \wh\mcP_{n(1+ T)}\cap N_r(\mcT_3)\big)\\ &= \tau^*_{2h}\big(\wh\mcP_{n(1+ T)}\cap N_r(\mcT_4)\big) = 1 \big\}
\end{align*} 
can be split as follows:
\[
\begin{split}
    &A^* = \left(\hspace{-3em}\bigcap_{\hspace{3em}x\in \left(\mcT^r_1\cup \mcT^r_2\cup \mcT^r_{1,2}\right)\cap\wh\mcP_{n(1+ T)}}\hspace{-5em} \{B_x \le h < B_x+L_x\}^c\right) \cap \left(\hspace{-3em}\bigcap_{\hspace{3em}x\in \left(\mcT^r_3\cup \mcT^r_4\cup \mcT^r_{3,4}\right)\cap\wh\mcP_{n(1+ T)}}\hspace{-5em} \{B_x \le 2h < B_x+L_x\}^c\right)\\
    &\qquad\quad\cap\left(\hspace{-9em}\bigcap_{\hspace{9em}x \in\left( \mcT^r_{1,3}\cup \mcT^r_{1,4}\cup \mcT^r_{2,3}\cup\mcT^r_{2,4}\cup\mcT^r_{1,2,3}\cup\mcT^r_{1,2,4}\cup\mcT^r_{1,3,4}\cup\mcT^r_{2,3,4}\cup\mcT^r_{1,2,3,4}\right)\cap\wh\mcP_{n(1+ T)}}\hspace{-14.5em} \{B_x \le h < B_x+L_x\}^c \cap \{B_x \le 2h < B_x+L_x\}^c\right),
\end{split}
\]
where $^c$ denotes the complement  and $\mcT_{J}^r$ (similarly to $\mcT_J$ in \eqref{TJ}) is used as a short-hand notation for 
\[
    \mcT_{J}^r \coloneqq  \left(\bigcap_{j\in J} N_r\left(\mcT_{j}\right)\right)\cap \left(\bigcap_{j\not\in J} \big(N_r\left(\mcT_{j}\right)\big)^c\right).
\]
We observe that, for any $x$, 
\begin{align*}
    \pr\big[ \{B_x \le h < B_x+L_x\}^c \cap \{B_x \le 2h < B_x+L_x\}^c \big] = 1-\frac{2-e^{- h}}{1+ T}.
\end{align*}

Note that the number of points in each set $\mcT_{J}^r$ is Poisson distributed and their birth and death times are independent. Using the moment generating function formula for the Poisson distribution, we can calculate   
\begin{align*}
    \pr[A^*|\bar\mcT] &= e^{-n\bbQ\left[\bigcup\limits_{i=1}^4 N_r(\mcT_i) \right] -n(1-e^{- h})\bbQ\left[ \bigcup_{J\in\mc K} \mcT_J^r  \right]}, 
\end{align*}
where 
\[ \mc K\coloneqq\big\{\{1,3\},\{1,4\},\{2,3\},\{2,4\},\{1,2,3\},\{1,2,4\},\{1,3,4\},\{2,3,4\},\{1,2,3,4\} \big\}.\]
Using the independence of $\tau$ and $\tau^*$, along with \eqref{Egkr2}, we have
\begin{align*}
    \left|\ex[g_h(\bar{\mcT},\wh\mcP_{n(1+ T)})]\right| &\leq \ex\left[\left|g_h(\bar{\mcT},\wh\mcP_{n(1+ T)})\right|\right] \\  &= \ex\left[\frac{\prod_{i=1}^4 \wt\xi_{k,r}(\mcT_i)}{(1+ T)^{|\cI|}} \left|\sum_{j=1}^8 E_1^j(h)D_1^j(r,h)  - E_2^j(h)D_2^j(r,h) \right|  \right],
\end{align*} where
$E_1^j(h)$ and $E_2^j(h)$, are the terms in \eqref{E1=} and \eqref{E2=}, respectively, and
\begin{align}
    D_\ell^j(r,h) &\coloneqq e^{-n\bbQ\left[\bigcup\limits_{i=1}^4 N_r(\mcT_i) \right] -n(1-e^{- h})\bbQ\left[ \bigcup_{J\in\mc K^h_{\ell,j}} \mcT_J^r  \right]  -n(1-e^{-2 h})\bbQ\left[ \bigcup_{J\in\mc K^{2h}_{\ell,j}} \mcT_J^r  \right]  }, \label{D1j} 
\end{align}
for suitably defined $\mc K^h_{\ell,j}$ and $\mc K^{2h}_{\ell,j}$. Note that
$   E_1^j(0) = E_2^j(0) = 1$, and
$    D_1^j(r,0)=D_2^j(r,0) = e^{-n \bbQ\left[N_r(\bar\mcT) \right]}$, where we use the shorthand notation,
\begin{align*}
    \bbQ\left[N_r(\bar\mcT) \right] \coloneqq \bbQ\left[\bigcup\limits_{i=1}^4 N_r(\mcT_i) \right].
\end{align*}
 Therefore,
 \begin{align}
    &\frac{\del}{\del h}\sum_{j=1}^8 E_1^j(h)D_1^j(r,h)  - E_2^j(h)D_2^j(r,h)\bigg|_{h=0} \notag\\
    &\qquad = \sum_{j=1}^8 \left(\frac{\del D_1^j(r,h)}{\del h}\bigg|_{h=0}  - \frac{\del D_2^j(r,h)}{\del h}\bigg|_{h=0} \right) + \sum_{j=1}^8 \left(\frac{\del E_1^j(h)}{\del h}\bigg|_{h=0}  - \frac{\del E_2^j(h)}{\del h}\bigg|_{h=0} \right). \notag
 \end{align}

Note that the second sum here
is equal to $\Phi'_\cI(0)=0$, and using the symmetry between the definitions of $E_{1,2}^j(h)$ and $D_{1,2}^j(r,h)$, it can be verified that the first sum is zero as well. Therefore, using the Taylor's theorem as in \eqref{PhicI}, it follows that there exists a $K_\cI(r, \bar{\mcT})$ for which 
\begin{align*}
    \sum_{j=1}^8 \left(E_1^j(h)D_1^j(r,h)  - E_2^j(h)D_2^j(r,h)\right) \leq K_\cI(r, \bar{\mcT}) h^2
\end{align*}
and an $h^*\in[0,h]$ such that
\begin{align*}
    K_\cI(r,\bar{\mcT}) = \frac{\del^2}{\del h^2}\sum_{j=1}^8 E_1^j(h)D_1^j(r,h)  - E_2^j(h)D_2^j(r,h)\bigg|_{h=h^*}.
\end{align*}
From \eqref{D1j}, we note that 
\begin{align*}
     \frac{\del D_\ell^j(r,h)}{\del h} &\leq e^{-n\bbQ\left[N_r(\bar\mcT) \right]} \left( n \bbQ\Bigg[ \bigcup_{J\in\mc K^h_{\ell,j}} \mcT_J^r\Bigg] + 2n\bbQ\Bigg[ \bigcup_{J\in\mc K^{2h}_{\ell,j}} \mcT_J^r  \Bigg]\right) \\
    &\leq 3 e^{-n\bbQ\left[N_r(\bar\mcT) \right]} n\bbQ\left[N_r(\bar\mcT) \right] \\
    &\leq 3e^{-1},
\end{align*}    
and
\begin{align*}     \frac{\del^2 D_\ell^j(r,h)}{\del h^2} &\leq 4e^{-n\bbQ\left[N_r(\bar\mcT) \right]} n\bbQ\left[N_r(\bar\mcT) \right]\left(n\bbQ\left[N_r(\bar\mcT) \right]+1\right)\leq 4,
\end{align*}
for all $h\geq 0$, $\ell\in\{1,2\}$, and $j\in \{1,\ldots, 8\}$. Similar bounds hold for the first two derivatives of $E_1^j(h)$ and $E_2^j(h)$. Therefore,
\begin{align*}
   K_\cI(r,\bar{\mcT}) \leq K_\cI  
\end{align*}
 almost surely, for some finite constant $K_\cI$, completing the proof. 
\end{proof}

Finally, we prove the corresponding bound for $\overline g_h(\bar{\mcT},\wh\mcP_{n(1+ T)})$, as follows. 
\begin{lemma}\label{lem:F2mom}
Let the setting and notation be as in \reflem{lem:PhiF}, and $\overline g_h$ as defined in \eqref{tilgdef}. Then
\[
\left|\ex\big[\overline g_h(\bar{\mcT},\wh\mcP_{n(1+ T)})\big]\right| \leq \frac{K_\cI}{(1+ T)^{|\cI|}}\,   \ex\big[\xi_{k,r}(\mcT_1) \xi_{k,r}(\mcT_2) \zeta_r(\mcT_1, \mcT_{2}) \big]\, h, 
\]
for some finite constant $K_\cI$.
\end{lemma}
\begin{proof}
As in the proofs of previous two lemmas, we can write
\begin{align*}
    \ex\big[\overline g_h(\bar{\mcT},\wh\mcP_{n(1+ T)})\big] = \ex\big[\xi_{k,r}(\mcT_1) \xi_{k,r}(\mcT_2)\zeta_r(\mcT_1, \mcT_{2})(\check{E}_1-\check{E}_2)\big],
\end{align*}
with 
\begin{align*}
    \check{E}_1\coloneqq &\tau_h(\mcT_1) \tau^*_h\big(\wh\mcP_{n(1+ T)}\cap N_r(\mcT_1)\big) \tau_h(\mcT_2) \tau^*_h\big(\wh\mcP_{n(1+ T)}\cap N_r(\mcT_2)\big)\\ & + \tau_0(\mcT_1) \tau^*_0\big(\wh\mcP_{n(1+ T)}\cap N_r(\mcT_1)\big) \tau_0(\mcT_2) \tau^*_0\big(\wh\mcP_{n(1+ T)}\cap N_r(\mcT_2)\big)\\
   \check{E}_2\coloneqq &\tau_0(\mcT_1)\tau^*_0\big( \wh\mcP_{n(1+ T)}\cap N_r(\mcT_1)\big)\tau_h(\mcT_2)\tau^*_h\big(\wh\mcP_{n(1+ T)}\cap N_r(\mcT_2)\big)\\ & + \tau_h(\mcT_1)\tau^*_h\big(\wh\mcP_{n(1+ T)}\cap N_r(\mcT_1)\big)\tau_0(\mcT_2)\tau^*_0\big(\wh\mcP_{n(1+ T)}\cap N_r(\mcT_2)\big).
\end{align*}

We repeat the same ideas in the proof of \reflem{lem:PhiF}. Note this time, due to symmetry, it is enough to calculate the first term of $\check{E}_2$,
\begin{align*}
    &\pr\left[\tau^*_0\big(\wh\mcP_{n(1+ T)}\cap N_r(\mcT_1)\big)=\tau^*_h\big( \wh\mcP_{n(1+ T)}\cap N_r(\mcT_2)\big)=1 \Big|\bar\mcT \right]\\
    &\qquad\qquad=\pr \left[\left(\hspace{-1em}\bigcap_{\hspace{1em}x\in \mcT^r_1 \cap\wh\mcP_{n(1+ T)}}\hspace{-2em} \{B_x > 0 \}\right) \cap \left(\hspace{-1em}\bigcap_{\hspace{1em}x\in \mcT^r_2\cap\wh\mcP_{n(1+ T)}}\hspace{-2em} \{B_x \le h < B_x+L_x\}^c\right)\right.\\
    &\phantom{\pr\left[ \left(\hspace{-1em}\bigcap_{\hspace{1em}x\in \mcT^r_1 \cap\wh\mcP_{n(1+ T)}}\hspace{-2em} \{B_x > 0 \}\right)\right. } \left.\cap\left(\hspace{-1em}\bigcap_{\hspace{1em}x \in \mcT^r_{1,2}\cap\wh\mcP_{n(1+ T)}}\hspace{-2em} \{B_x >0\} \cap \{B_x \le h < B_x+L_x\}^c\right)\Bigg|\bar\mcT\right]\\
    &\qquad\qquad= e^{-n\bbQ\left[ N_r(\mcT_1)\cup N_r(\mcT_2)\right] -n(1-e^{- h})\bbQ\left[ \mcT_{1,2}^r  \right]  }.
\end{align*}
Furthermore,
\begin{align*}
    \pr\left[\tau_0(\mcT_1) =  \tau_h(\mcT_2)=1\right] &= \pr \left[\left(\bigcap_{x\in \mcT_1 } \{B_x = 0 \}\right) \cap \left(\bigcap_{x\in \mcT_2} \{B_x \le h < B_x+L_x\}\right)\right.\\
    & \hspace{6em}\left.\cap\left(\bigcap_{x \in \mcT_{1,2}} \{B_x=0 \le h < B_x+L_x\}\right)\right]\\
    &= \left(\frac{1}{1+ T}\right)^{|\cI|} e^{- h I_{1,2} }.
\end{align*}
Therefore,
\begin{align*}
    &\ex\big[\overline g_h(\bar{\mcT},\wh\mcP_{n(1+ T)})\big] = (1+ T)^{-|\cI|} \nn \\ &\qquad\qquad\times\ex\left[\xi_{k,r}(\mcT_1) \xi_{k,r}(\mcT_2) \zeta_r(\mcT_1, \mcT_{2}) e^{-n\bbQ\left[N_r(\bar\mcT)\right]} \left( 1 - e^{- h I_{1,2} } e^{ -n (1-e^{- h})\bbQ\left[ \mcT_{1,2}^r  \right]  } \right)\right].
\end{align*}
Noting that, for $h=0$, the expression inside the expectation is zero,  
and 
\begin{align*}
    \frac{\del}{\del h} e^{-n\bbQ\left[N_r(\bar\mcT)\right]} \left(1 - e^{- h I_{1,2} }e^{ -n (1-e^{- h})\bbQ\left[ \mcT_{1,2}^r  \right]  } \right)\Bigg|_{h=h^*}&\leq  e^{-n\bbQ\left[N_r(\bar\mcT)\right]}\left(I_{1,2} + n \bbQ\left[ \mcT_{1,2}^r  \right]\right)\\
    &\leq I_{1,2} + e^{-n\bbQ\left[N_r(\bar\mcT)\right]} n\bbQ\left[N_r(\bar\mcT)\right] \\
    &\leq I_{1,2} + e^{-1}
\end{align*}
for all $h^*\geq 0$. Therefore, applying Taylor's theorem, we conclude that almost surely
\begin{align*}
    e^{-n\bbQ\left[N_r(\bar\mcT)\right]} \left(1 - e^{- h I_{1,2} }e^{ -n (1-e^{- h})\bbQ\left[ \mcT_{1,2}^r  \right]  } \right) \leq K_\cI' h,  
\end{align*}
for some constant $K_\cI'$. This completes the proof. 
\end{proof}

\section{}\label{alphaproof}
\sec
In this section we prove the asymptotic identities \eqref{alphaasymp} used in the proof of \reflem{covgen2}.
\begin{proof}[Proof of \eqref{alphaasymp}]
Note that 
\begin{align}
    \alpha_{k,0}^\Delta- \alpha_{k,0}^\infty =   \ex\big[\xi_{k,r}(\mb{X}) 
    \xi_{k,r}(\mb{X}')\nu_r(\mb{X},\mb{X}')  \big], \label{alphdiff}
\end{align}
where $\mb X$ and $\mb X'$ are each  a $k$-tuple of iid points, as defined before, and
\begin{align}
    \nu_r(\mb{x};\mb{x}') \coloneqq e^{-n\left(\bbQ\left[N_r(\mb{x})\right]+\bbQ\left[ N_r(\mb{x}')\right]\right) +ne^{-\Delta}\bbQ\left[N_r(\mb{x})\cap N_r(\mb{x}')\right]}
    - e^{-n\left(\bbQ\left[N_r(\mb{x})\right]+\bbQ\left[ N_r(\mb{x}')\right]\right)}. \label{tildezeta}
\end{align}
 Note that $\nu_r$ is not translation and scale invariant, and so \reflem{lemgeom} is not directly applicable for calculating the asymptotics of \eqref{alphdiff}. However, we can adapt the proof of \reflem{lemgeom} to get the desired result, as follows. By a change of variables, and using the translation and scale invariance of $\xi_{k,r}$ we can write,
 \begin{align}
 \begin{split}
     \frac{\alpha_{k,0}^\Delta- \alpha_{k,0}^\infty}{r^{d(2k-1)}} = &\int_{(\Rd)^{2k}}  \xi_{k,1}(0,\mb{y}) 
    \xi_{k,1}(\mb{y}')
    \nu_r\big(x+r(0,\mb{y});x+r \mb{y}'\big) q(x)\der x\\&\times\prod_{i=1}^{2k-1}q(x+ry_i)\der y_i.
    \end{split}\label{intrdk}
 \end{align}
 for $(0,\mb y) = (0,y_1,\ldots,y_{k-1})$ and $\mb y' = (y_{k},\ldots, y_{2k-1})$. Now note that, if 
 \[N_r\big(x+r(0,\mb{y})\big)\cap N_r\big(x+r\mb{y}'\big) = \varnothing,
 \]
 then
 \[\nu_r\big(x+r(0,\mb{y}); x+r\mb{y}'\big)=0.\] 
 However, due to affine invariance of $N_r$, the former is equivalent to $N_1(0,\mb{y})\cap N_1(\mb{y}') = \varnothing$. Using this observation, the localness of $\xi_{k,1}$, and the fact that $0\leq\nu_r\leq 1$ almost surely, the integrand in \eqref{intrdk} can be  bounded above by
 \begin{align*}
     \|\xi_k\|_{\infty}^2 \big(\|q\|_{\infty}\big)^{2k}  \ind \big\{\{y_1,\ldots,y_{2k-1}\} \subset B\big(0,2\beta_k\max(\rmax, 1)\big)\big\},
 \end{align*}
 which is integrable. Therefore, the  dominated convergence theorem is applicable to the integration in \eqref{intrdk}. Note that, due to the affine invariance of $N_r$,
 \begin{align*}
     \bbQ\big[N_r\big(x+r(0,\mb{y})\big)\big] &=  \int_{\Rd} q(u)\ind\big\{u\in x+rN_1(0, \mb{y})\big\}\der u\nn\\
     & =  r^d \int_{\Rd} q(x+rz)\ind\big\{z\in N_1(0, \mb{y})\big\}\der z,
 \end{align*}
 where we used the change of variable, $u=x+rz$, again.
 Using the dominated convergence theorem, and $\lim\limits_{n\to\infty}nr^d=\gamma$,
 \begin{align*}
     \lim_{n\to \infty} n \bbQ\big[N_r\big(x+r(0,\mb{y})\big)\big] = \gamma q(x) \vol[N_1(0, \mb{y})].
 \end{align*}
 Similarly,
  \[   \lim_{n\to \infty} n \bbQ\big[N_r\big(x+r\robert{\mb{y}'})\big)\big] = \gamma q(x) \vol[N_1(\mb{y}')],
     \]
     and
     \[ \lim_{n\to \infty} n \bbQ\big[N_r\big(x+r(0,\mb{y})\big)\cap N_r\big(x+r\robert{\mb{y}'}\big) \big] = \gamma q(x) \vol\left[N_1(0, \mb{y})\cap N_1(\mb{y}')\right].
 \]
 From \eqref{tildezeta},
 \begin{align*}
     \lim_{n\to \infty} \nu_r\big(x+r(0,\mb{y}); x+r\robert{\mb{y}'}\big)=&e^{-\gamma q(x) \left(\vol\left[N_1(0,\mb{y})\right] + \vol\left[N_1(\mb{y}')\right]\right)}\nn\\
     &\times\left(e^{\gamma q(x) e^{-\Delta} \vol\left[N_1(0,\mb{y})\cap N_1(\mb{y}')\right]}-1\right).
 \end{align*}
 Therefore, from \eqref{intrdk}, applying  dominated convergence and the almost everywhere continuity of $q$,
\begin{align*}
\begin{split}
    \lim_{n\to\infty}r^{-d(2k-1)}(\alpha_{k,0}^\Delta- \alpha_{k,0}^\infty)
    &\hspace{-2pt}=  \hspace{-6pt}\int_{(\Rd)^{2k}}\xi_{k,1}(0,\mb{y})\xi_{k,1}(\mb{y}') [q(x)]^{2k}
    e^{-\gamma q(x) \left(\vol\left[N_1(0,\mb{y})\right] + \vol\left[N_1(\mb{y}')\right]\right)}
    \\
    &  \times\left(e^{\gamma q(x) e^{-\Delta} \vol\left[N_1(0,\mb{y})\cap N_1(\mb{y}')\right]}-1\right)\der x \prod_{i=1}^{2k-1} \der y_i.
    \end{split}
\end{align*}
Replacing the second exponential term  with its power series expansion, interchanging the sum and the integral by Fubini, and setting 
\robert{
\begin{align*}
\begin{split}
    \kappa_{k,0,l} \coloneqq &\int_{(\Rd)^{2k}} \xi_{k,1}(0,\mb{y}) \xi_{k,1}(\mb{y}')  [q(x)]^{2k+l} \\
    &\times e^{-\gamma q(x)\left(\vol\left[N_1(0,\mb{y})\right] + \vol\left[N_1(\mb{y}')\right]\right)}
      \left(\vol\left[N_1(0,\mb{y})\cap N_1(\mb{y}')\right]\right)^l \der x\prod_{i=1}^{2k-1} \der y_i,
      \end{split}
\end{align*}
 we establish the first part of \eqref{alphaasymp}.
In particular, note that
\begin{align}
\begin{split}\label{kappak0l}
     \kappa_{k,0,l} &\geq e^{-2\gamma V_d [\beta_k\max(1,\rmax)]^d \|q\|_{\infty}} \int_{\Rd}[q(x)]^{2k+l}\der x \\
     &\times \int_{(\Rd)^{2k-1}} \xi_{k,1}(0, \mb y) \xi_{k,1}(\mb y') \left(\vol\left[N_1(0,\mb{y})\cap N_1(\mb{y}')\right]\right)^l \prod_{i=1}^{2k-1}\der y_i,
     \end{split}
\end{align}
\robert{where $V_d$ denotes the Lebesgue volume of the unit ball in $\Rd$.
The latter term is positive
due to Assumptions \ref{localass}, \ref{loc2} and \ref{jointfeas}. }
Proving the second part of \eqref{alphaasymp}, follows exactly the same steps, by taking
\begin{align}\label{kappakjl}
\begin{split}
    \kappa_{k,j,l} \coloneqq &\int_{(\Rd)^{2k-j}} \xi_{k,1}(0,\mb{y}) \xi_{k,1}(\mb{y}')  [q(x)]^{2k-j+l} \\
    &\times e^{-\gamma q(x)\left(\vol\left[N_1(0,\mb{y})\right] + \vol\left[N_1(\mb{y}')\right]\right)}
      \left(\vol\left[N_1(0,\mb{y})\cap N_1(\mb{y}')\right]\right)^l \der x\prod_{i=1}^{2k-1-j} \der y_i.
      \end{split}
\end{align}
We therefore omit the details here.
}
\end{proof}

\end{appendix}


\bibliographystyle{imsart-number}
\bibliography{references}       

\end{document}